%% file: main.tex
\documentclass[12pt,oneside]{amsart}
\usepackage{amsmath,amssymb,cite,mathrsfs,tikz-cd}
\usepackage[all]{xy}
\usepackage{typearea} 
\usepackage{hyperref} 
\usepackage{color} 

\input{macros.tex}

\newcommand{\pullbackcorner}[1][dr]{\save*!/#1-1.7pc/#1:(-1.5,1.5)@^{|-}\restore}

\begin{document}
\title{The relative Manin--Mumford conjecture}
\author{Ziyang Gao}
\author{Philipp Habegger}

\address{Department of Mathematics, UCLA, Los Angeles, CA 90095, USA}
\email{ziyang.gao@math.ucla.edu}
\address{Department of Mathematics and Computer Science; University of Basel; Spiegelgasse 1, 4051 Basel, Switzerland}
\email{philipp.habegger@unibas.ch}


\maketitle

{\centering\footnotesize \textit{To the memory of Bas Edixhoven}\par}

\begin{abstract}
  We prove the Relative Manin--Mumford Conjecture for
  families of abelian varieties in characteristic $0$. We follow the
  Pila--Zannier method to study special point problems, and we use the Betti map which 
  goes back to work of Masser and Zannier in the case of curves. The key
  new ingredients compared to previous applications of this approach are
  a height inequality proved by both authors of the current paper and
  Dimitrov, and the first-named author's study of certain degeneracy
  loci in subvarieties of abelian schemes.
  
  We also strengthen this result and prove a criterion for torsion points to be dense in a subvariety of an abelian scheme over $\IC$.

  The Uniform Manin--Mumford
  Conjecture for curves embedded in their Jacobians was first proved
  by K\"uhne. We give a new  proof, as a corollary to our main theorem, that does
  not use equidistribution.  
\end{abstract}
\tableofcontents

\section{Introduction}
\input{intro.tex}

\section{Bi-algebraic Structure on the Universal Abelian Variety}
\label{SectionBiAlg}
\input{BiAlgUnivAb.tex}

\section{The Degeneracy Locus}
\label{secdeglocus}
\input{deglocus.tex}

\section{Large Galois Orbits}
\label{SectionLGO}
\input{Galorbit.tex}

\section{Point Counting}
\input{PointCount.tex}

\section{Application of Mixed Ax--Schanuel}
\label{SectionAppAS}
\input{AppAS.tex}

\section{The Relative Manin--Mumford Conjecture over $\IQbar$}
\label{SectionConclusion}
\input{RMM.tex}

\section{Uniformity of the Number of Algebraic Torsion Points in
  Curves}
\label{SectionUnifMM}
\input{UnifMM.tex}

\section{A Criterion for Zariski density over $\IQbar$}
\label{SectionCriterionTorsionDenseOverIQbar}
\input{BettiSubmersive.tex}

\section{Specialization: From $\IQbar$ to $\IC$}
\label{SectionSpecialization}
\input{Specialization.tex}
\appendix
\renewcommand{\thesection}{\Alph{section}}
\setcounter{section}{0}

\section{Large Galois Orbits revisited}
\label{Appendix}
\input{Appendix.tex}

\bibliographystyle{alpha}
\bibliography{literature}

\end{document}

%% file: macros.tex
\newcounter{maincounter}
\numberwithin{maincounter}{section}
\numberwithin{equation}{section}

\newtheorem{lemma}[maincounter]{Lemma}
\newtheorem{proposition}[maincounter]{Proposition}
\newtheorem{corollary}[maincounter]{Corollary}

\newtheorem{theorem}[maincounter]{Theorem}

\newtheorem{definition}[maincounter]{Definition}

\def\AA{\mathbb{A}}

\def\NN{\mathbb{N}}
\def\RR{\mathbb{R}}

\def\CC{\mathbb{C}}
\def\ZZ{\mathbb{Z}}
\def\PP{\mathbb{P}}
\def\QQ{\mathbb{Q}}

\newcommand{\cal}{\mathcal}

\newcommand{\cA}{\cal{A}}

\newcommand{\cB}{\cal{B}}

\newcommand{\IP}{{\PP}}

\newcommand{\IC}{{\CC}}
\newcommand{\IR}{{\RR}}

\newcommand{\IRanexp}{{{\RR}_{\rm an,exp}}}

\newcommand{\IQbar}{{\overline{\QQ}}}

\newcommand{\IZ}{{\ZZ}}
\newcommand{\IN}{{\NN}}
\newcommand{\IA}{{\AA}}
\newcommand{\IQ}{{\QQ}}

\newcommand{\cL}{{\mathcal L}}
\newcommand{\cM}{{\mathcal M}}

\newcommand{\ssm}{\setminus}
\newcommand{\ord}[1]{{\rm ord}({#1})}

\renewcommand{\subset}{\subseteq} 
\renewcommand{\supset}{\supseteq}

\newcommand{\tor}[1]{{#1}_{\mathrm{tor}}}

%% file: intro.tex
Let $S$ be a regular, irreducible, quasi-projective
variety defined over an algebraically closed  field $L$ of characteristic $0$. Let $\pi
\colon \cA \rightarrow S$ be an abelian scheme of relative dimension
$g  \ge 1$, namely a proper smooth group scheme whose fibers are
abelian varieties.
Let $\cA_{\mathrm{tor}}$ denote the union over all $s\in S(L)$ of
the group of all torsion points in $\cA_s = \pi^{-1}(s)$. For each $N
\in \mathbb{Z}$, let $[N] \colon \cA \rightarrow \cA$ be the
multiplication-by-$N$ morphism.

The goal of this paper is to prove the \textit{relative Manin--Mumford
  conjecture} for abelian schemes.

\begin{theorem}
  \label{MainThm}
  Let $X$ be an irreducible subvariety of $\cA$. Assume that
  $\mathbb{Z}X := \bigcup_{N\in \mathbb{Z}} [N]X$ is Zariski dense in
  $\cA$.
  If $X(L) \cap \cA_{\mathrm{tor}}$ is Zariski dense in $X$, then
  $\dim X \ge g$.
\end{theorem}
Throughout the whole paper, by \textit{irreducible subvariety} we mean \textit{closed irreducible subvariety} unless stated otherwise.

The assumption $\ZZ X$ being Zariski dense in $\cA$ can be checked
over the geometric generic fiber of $\cA \rightarrow S$. Indeed, let
$\eta$ be the generic point of $S$ and fix an algebraic closure of the
function field of $S$. Write $X_{\overline{\eta}}$ for the geometric
generic fiber of $\pi|_X$.
Then $X_{\overline \eta}$ is non-empty if and only if $\pi|_X\colon
X\rightarrow S$ is dominant.
In particular, $\cA_{\overline{\eta}}$ is an abelian variety over
an algebraically closed field containing the possible reducible $X_{\overline\eta}$. Then $\ZZ X$ is Zariski dense in $\cA$
if and only if $X_{\overline{\eta}}$ is non-empty and not contained in
a finite union of proper algebraic subgroups of $\cA_{\overline{\eta}}$.

The Relative Manin--Mumford Conjecture was inspired by S.~Zhang's ICM
talk \cite{zhang1998small} and proposed by Pink \cite[Conjecture~6.2]{Pink} and 
Zannier~\cite{ZannierBook}. In the case $\dim X=1$ it was proved in a
series of papers by Masser--Zannier and Corvaja--Masser--Zannier
\cite{MZ:torsionanomalous,MasserZannierTorsionPointOnSqEC,
MASSER2014116, MasserZannierRelMMSimpleSur, CorvajaMasserZannier2018,
MasserZannierRMMoverCurve}. See also work of
Stoll~\cite{Stoll:simtorsion} for an explicit case. For surfaces some 
results are due to the first-named author~\cite{hab:weierstrass} and the recent work of 
Corvaja--Tsimerman--Zannier~\cite{CTZ:23}. When $\cA$ is a fibered product of families of elliptic curves, it was proved by K\"{u}hne \cite{KuehneRBC}. A core idea of many of these papers, as in ours, is to 
use the Betti coordinates introduced by Masser and Zannier~\cite{MZ:torsionanomalous} to study questions in diophantine
geometry.

As a corollary  of Theorem~\ref{MainThm} we obtain  the \textit{Uniform
Manin--Mumford Conjecture} for curves embedded in their Jacobians, which was recently proved by K\"{u}hne
\cite[Theorem~1.2]{KuehneUnifMM}.

\begin{corollary}\label{CorUniformMM}
  For each integer $g \ge 2$, there exists a constant $c = c(g) > 0$ with the following property. 
  Let $C$ be an irreducible, smooth, projective curve of genus $g$
  defined over $\IC$. Let $x_0 \in C(\IC)$, and let $C-x_0$ be the
  image of the Abel--Jacobi embedding based at $x_0$ in the Jacobian
  $\mathrm{Jac}(C)$ of $C$. Then
  \begin{equation}
    \#(C(\IC)-x_0) \cap \mathrm{Jac}(C)_{\mathrm{tor}} \le c.
  \end{equation}
\end{corollary}

A second proof of the Uniform Manin--Mumford Conjecture for curves was given by Yuan in \cite{YuanArithBig}, based on the theory of adelic line bundles over 
quasi-projective varieties of Yuan--Zhang \cite{YuanZhangEqui}. Prior to K\"{u}hne's 
proof of the full conjecture, DeMarco--Krieger--Ye \cite{DeMarcoKriegerYeUniManinMumford} proved the case where $g = 2$ and $C$ is bi-elliptic, using method of 
arithmetic dynamical systems.

We hereby give a different proof. The
common tools used in \cite{KuehneUnifMM} and the current paper are the
height inequality \cite[Theorems 1.6 and B.1]{DGHUnifML} proved by both
authors of the current paper and Dimitrov, and the first-named
author's results on the generic rank of the Betti map \cite{GaoBettiRank}. K\"{u}hne used this height inequality,
among other tools, to prove an equidistribution result. 
He then applied this equidistribution result and 
\cite[Theorem 1.3]{GaoBettiRank} in combination with the Ullmo--S.~Zhang
approach to the Bogomolov Conjecture to conclude the argument.

Our proof does not involve equidistribution, and we apply
\cite{GaoBettiRank} in a different way.  The crucial case for
Theorem~\ref{MainThm} is the case where $L$ is an algebraic closure
$\IQbar$ of $\IQ$, \textit{i.e.}, in the setup of the theorem we assume every variety is defined over $\IQbar$. The proof of Theorem~\ref{MainThm} over $\IQbar$ occupies the current paper up to $\mathsection$\ref{SectionConclusion}. Then in
$\mathsection$\ref{SectionUnifMM} we prove
Corollary~\ref{CorUniformMM} as a consequence of
Theorem~\ref{MainThm} over $\IQbar$. The deduction is inspired by \cite[Theorem~2.4]{Stoll:Uniform} and \cite[Proposition~7.1]{DGHUnifML}.

Our proof of Theorem~\ref{MainThm} for  $L=\IQbar$ is in spirit of the Pila--Zannier
method~\cite{PilaZannier} to solve special point problems. Roughly
speaking this strategy can be divided into four steps. We start  
with a 
large Galois orbit result on 
torsion points on an abelian variety,
quantifying earlier work of Masser~\cite{Masser:smallvalues}. Such a
result can be deduced from work of David~\cite{DavidMinHaut}. It
also follows 
 more directly by R\'{e}mond and Gaudron's refinement
 \cite{RemondGaloisBound, GRGaloisBound}  of deep work of Masser and
 W\"{u}stholz
 on isogeny estimates for abelian varieties \cite{MW:abelianisog}. 
A key new input at this step is the 
height
inequality \cite[Theorem~B.1]{DGHUnifML} which roughly speaking allows
us to bound the height of the abelian variety itself.
Then we introduce a suitable set that is definable in an o-minimal
structure, here $\IRanexp$, that encodes our points of interest as rational.
We then use this result to invoke an
appropriate version of the Pila--Wilkie counting theorem
due to the second-named author and Pila \cite[Corollary~7.2]{HabeggerPilaENS}.
Finally, we apply a suitable functional transcendence theorem, the
mixed Ax--Schanuel theorem proved by the first-named author
\cite{GaoMixedAS}. In this last step we study the degeneracy locus,
defined in $\mathsection$\ref{secdeglocus}, as was done in
\cite[Proposition~1.10]{GaoBettiRank}.
 
Our application of the height
inequality \cite[Theorem~B.1]{DGHUnifML} differs  from the approaches in
\cite{DGHUnifML} and \cite{KuehneUnifMM}. Instead of constructing a
non-degenerate subvariety, we study the degeneracy loci more
carefully. We prove the desired result by dividing into two cases:
either the height inequality \cite[Theorem~B.1]{DGHUnifML} is
applicable or not. If it is not applicable, then the $0$-th degeneracy
locus is large and we may apply \cite[Proposition~1.10]{GaoBettiRank}.
If it is applicable, then we follow the Pila--Zannier method described
above. Ultimately we show that 
\cite[Proposition~1.10]{GaoBettiRank} still applies and can conclude
the Relative Manin--Mumford Conjecture.

\subsection{Criterion of torsion points being dense}
Let us assume for the moment that the base field $L$ is a subfield of $\IC$.

It is natural to ask whether the following converse of
Theorem~\ref{MainThm} is true.
Let $X\subset \cA$ be an irreducible subvariety.
If $\ZZ X$ is Zariski
dense in $\cA$ and $\dim X \ge g$, then is it true that $X(\IC) \cap
\cA_{\mathrm{tor}}$ is Zariski dense in $X$? This question is
related to the generic Betti rank, as is studied in \cite{ACZBetti}.
Indeed, using \cite[Proposition~2.1.1]{ACZBetti}, one can show that the
answer is yes \textit{in some cases}, even for the Euclidean topology. For example three cases were proved in \cite{ACZBetti}: if $g = 2, 3$ and $\cA/S$ has no fixed part over any finite covering of $S$, if 
the geometric generic fiber $\cA_{\overline{\eta}}$ satisfies 
$\mathrm{End}(\cA_{\overline{\eta}}) = \ZZ$, and if $\cA \rightarrow
S$ is the Jacobian of the universal hyperelliptic curve.
\cite[Theorem~1.4.(i)]{GaoBettiRank} proves some more general cases,
for example if $\cA_{\overline{\eta}}$ is simple. Some new cases for
the denseness in the Euclidean topology were
recently proved by Eterovi\'{c}--Scanlon in \cite{EterovicScanlon}.

In this paper, we provide a criterion for $X(\IC) \cap \cA_{\mathrm{tor}}$ being Zariski dense in $X$. %
Using this criterion, one can show that the converse of
Theorem~\ref{MainThm} is \textit{false in general} even if $\cA/S$ has
no fixed part over any finite covering of $S$. See  \cite[Example~9.4]{GaoBettiRank} for a counterexample with $g = 4$. 
 The statement of this criterion, Theorem~\ref{ThmCritTorsionDense}, involves the Betti map and the Betti
 rank. We briefly recall the definition here. For a precise definition of the Betti map, we refer to \cite[$\mathsection$3-$\mathsection$4]{GaoBettiRank} or \cite[$\mathsection$2.3 and $\mathsection$B.1]{DGHUnifML}.

For any $s \in S(\IC)$, there exists a simply-connected open neighborhood $\Delta
\subseteq S^{\mathrm{an}}$ of $s$; here and below the superscript
``$\mathrm{an}$'' refers to complex analytification. Then one can define a basis $\omega_1(s),\ldots,\omega_{2g}(s)$ of the period lattice of each fiber $s \in \Delta$ as holomorphic functions of $s$. Now each fiber $\cA_s = \pi^{-1}(s)$ can be identified with the complex torus $\IC^g/\left(\IZ \omega_1(s)\oplus \cdots \oplus \IZ\omega_{2g}(s)\right)$, and each point $x \in \cA_s(\IC)$ can be expressed as the class of $\sum_{i=1}^{2g}b_i(x) \omega_i(s)$ for real numbers $b_1(x),\ldots,b_{2g}(x)$. Then $b_{\Delta}(x)$ is defined to be the class of the $2g$-tuple $(b_1(x),\ldots,b_{2g}(x)) \in \IR^{2g}$ modulo $\IZ^{2g}$.
We obtain a real-analytic map $
b_{\Delta} \colon \cA_{\Delta} = \pi^{-1}(\Delta) \rightarrow \mathbb{T}^{2g}$, 
which is fiberwise a group isomorphism and
where $\mathbb{T}^{2g}$ is the real torus of dimension $2g$. The map
$b_{\Delta}$ is well-defined up-to real analytic automorphisms of
$\mathbb{T}^{2g}$, \textit{i.e.}, elements of $\mathrm{GL}_{2g}(\ZZ)$.

Let $X^{\mathrm{reg}}$ denote the regular locus of $X$.
The \textit{generic Betti rank} of $X$ is defined in terms of the
differential
$\mathrm{d} b_\Delta$ to be
\[
\mathrm{rank}_{\mathrm{Betti}}(X) := \max_{x \in X^{\mathrm{reg,an}}\cap \cA_{\Delta}} \mathrm{rank}_{\RR} (\mathrm{d}b_{\Delta}|_{X\cap \cA_{\Delta}})_x.
\]
This definition does not depend on the choice of $\Delta$ (see \cite[end of $\mathsection$4]{GaoBettiRank}) or $b_{\Delta}$ (since every two $b_{\Delta}$ differ from an element in $\mathrm{GL}_{2g}(\ZZ)$). The definition of the generic Betti rank yields the following trivial upper bound
\begin{equation}\label{EqBettiRank}
\mathrm{rank}_{\mathrm{Betti}}(X) \le 2 \min\{\dim X, g\}.
\end{equation}

\begin{theorem}
  \label{ThmCritTorsionDense}
  Assume $\ZZ X$ is Zariski dense in $\cA$. Then $X(L) \cap \cA_{\mathrm{tor}}$ is Zariski dense in $X$ if and only if $\mathrm{rank}_{\mathrm{Betti}}(X) = 2g$.
\end{theorem}

The ``if'' direction of this theorem is proved by
\cite[Proposition~2.1.1]{ACZBetti} and uses real analytic geometry. In fact, if $\mathrm{rank}_{\mathrm{Betti}}(X) = 2g$, then $X(\IC) \cap \cA_{\mathrm{tor}}$ is dense in the Euclidean topology on $X^{\mathrm{an}}$.

The ``only if'' direction of this theorem implies Theorem~\ref{MainThm} by \eqref{EqBettiRank}. In this paper, we prove this direction by a combination of Theorem~\ref{MainThm} and the criterion for $\mathrm{rank}_{\mathrm{Betti}}(X) = 2g$ given by \cite[Theorem~1.1]{GaoBettiRank}.

\subsection{From $\IQbar$ to $\IC$}
Early work of Bombieri--Masser--Zannier \cite{BMZUnlikely} shows how
to reduce certain Unlikely Intersection statements from base field $\IC$ to
$\IQbar$ using a specialization argument.
Barroero and Dill~\cite{BarDill:distinguised} provide a
 framework to approach these question in large generality. 
A specialization argument, based on a result of Masser, allows to pass
from $\IQbar$ to $\IC$ for Corollary~\ref{CorUniformMM}. However, the
specialization argument for Theorem~\ref{MainThm} and
Theorem~\ref{ThmCritTorsionDense} is more complicated.

All the varieties in question are defined over an algebraically closed
field of finite transcendence degree $d$ over $\IQbar$.
The proof of Theorem~\ref{MainThm} for $L=\IQbar$, \textit{i.e.}, if $d =
0$ is completed in $\mathsection$\ref{SectionConclusion}.
The proof of Theorem~\ref{ThmCritTorsionDense} is completed in
$\mathsection$\ref{SectionCriterionTorsionDenseOverIQbar}.

After this, we proceed by induction on $d$ by proving 
  two things in
$\mathsection$\ref{SectionSpecialization}: Theorem~\ref{MainThm} when
$\mathrm{trdeg}_\IQbar L\le d$ implies
Theorem~\ref{ThmCritTorsionDense} when
$\mathrm{trdeg}_\IQbar L\le d$, and
Theorem~\ref{ThmCritTorsionDense} when
$\mathrm{trdeg}_\IQbar L\le d$ implies
Theorem~\ref{MainThm} for
$\mathrm{trdeg}_\IQbar L\le d+1$.

Shortly before this paper was finalized, Corvaja, Tsimerman, and Zannier~\cite[Appendix~A]{CTZ:23} have independently
shown how to reduce Theorem~\ref{MainThm} from
$\IC$ to $\IQbar$ with a different argument.


\subsection{Notation} 
Let $\mathbb{A}_g$ be the moduli space of abelian varieties of
dimension $g$ with a polarization of type $\mathrm{diag}(d_1,\ldots,d_g)$
with $d_1\mid \cdots\mid d_g$ positive integers, endowed with
symplectic level-$\ell$-structure for some large enough but fixed
$\ell$ that is coprime to $d_g$ (so that $\mathbb{A}_g$ is a fine
moduli space).
Let $\pi\colon \mathfrak{A}_g \rightarrow \mathbb{A}_g$
be the universal abelian variety.
Both $\mathfrak{A}_g$ and $\IA_g$ are irreducible, regular,
quasi-projective varieties definable over a number field.

Any abelian variety is isogenous to a principally polarized abelian
variety, after extending the base field. We may apply this observation
to the generic fiber of our abelian schemes. Our main results
Theorems~\ref{MainThm} and \ref{ThmCritTorsionDense} are insensitive
to \'etale base change of $S$ and up-to isogeny. Therefore it is
possible to recover all results here while assuming $d_1=\cdots=d_g=1$
for the universal family.



\subsection*{Acknowledgements} Ziyang Gao has received
funding from the European Research Council (ERC) under the
European Union's Horizon 2020 research and innovation programme (grant
agreement n$^\circ$ 945714). Philipp Habegger has received funding
from the Swiss National Science Foundation (grant n$^\circ$ 200020\_184623). 
The authors would like to thank Daniel Bertrand and Umberto Zannier for having communicated with us the current formulation of the Relative Manin--Mumford Conjecture, in particular during the conference ``Third ERC Research Period on Diophantine Geometry'' in Rome 2015, and would like to thank the organizers of this conference. We would also like to thank Michael Stoll for comments on a previous version of the paper and Eric Gaudron for having pointed out to us the reference \cite{RemondGaloisBound}. We would like to thank the referees for their careful reading and valuable comments.


%% file: BiAlgUnivAb.tex
The goal of this section is to collect some basic facts about the universal abelian variety, especially its bi-algebraic structure. The end of this section contains the functional transcendence statement, the \textit{Ax--Schanuel theorem} for the universal abelian variety.

Let $d_1,\ldots,d_g$ be positive integers such that $d_1\mid \cdots
\mid d_g$. Let $D = \mathrm{diag}(d_1,\ldots,d_g)$ be the $g\times g$
diagonal matrix.

\subsection{Moduli space of abelian varieties} 
Let $\mathbb{A}_g$ be the moduli space of abelian varieties which are polarized of type $D$. In this subsection, we describe the bi-algebraic geometry on $\mathbb{A}_g$.

Let $\mathfrak{H}_g$ be the Siegel upper half space defined by
\[
\{ Z = X + \sqrt{-1}Y \in \mathrm{Mat}_{g \times g}(\mathbb{C}) : Z =
Z^{\!^{\intercal}},X,Y\in \mathrm{Mat}_{g\times g}(\mathbb{R}),~ Y~\text{is positive definite}\}.
\]
On identifying a complex number with its real and imaginary parts, 
 $\mathfrak{H}_g$ becomes a semi-algebraic subset of $\IR^{2g^2}$.
Note that $\mathfrak{H}_g$ is an open subset, in the Euclidean topology,  of
$M_{g\times g}(\IC) = \{Z \in \mathrm{Mat}_{g \times g}(\mathbb{C}) : Z =
Z^{\!^{\intercal}}\} \cong \IC^{g(g+1)/2}$. In fact,
$\mathfrak{H}_g$ is a connected 
complex manifold. 
It is well-known that the universal covering space of
$\mathbb{A}_g^{\mathrm{an}}$, the analytification of $\mathbb{A}_g$,
is given by $\mathfrak{H}_g$ and the covering map is a holomorphic map
\[
u_B \colon \mathfrak{H}_g \rightarrow  \mathbb{A}_g^{\mathrm{an}}.
\]


\begin{definition}
\begin{enumerate}
\item[(i)] A subset $\tilde{Y} \subseteq \mathfrak{H}_g$ is said to be
  irreducible algebraic if $\tilde{Y}$ is a complex  analytic
  irreducible component of $W \cap \mathfrak{H}_g$, where $W$ is an
  algebraic subset of $M_{g \times g}(\mathbb{C})$.
\item[(ii)] An irreducible subvariety $Y$ of $\mathbb{A}_g$ is said to be bi-algebraic if $Y = u_B(\tilde{Y})$ for some irreducible algebraic subset $\tilde{Y}$ of $\mathfrak{H}_g$.
\end{enumerate}
\end{definition}

The following notation will be used. Let $\tilde{Z}$ be a complex
analytic irreducible subset of $\mathfrak{H}_g$. We use
$\tilde{Z}^{\mathrm{Zar}}$ to denote the smallest algebraic subset of
$\mathfrak{H}_g$ that contains $\tilde{Z}$
and use $u_B(\tilde{Z})^{\mathrm{biZar}}$ to denote the smallest
bi-algebraic subvariety of $\mathbb{A}_g$ that contains
$u_B(\tilde{Z})$. We have the following
relation
\begin{equation*}
u_B(\tilde{Z})^{\mathrm{Zar}} \subseteq u_B(\tilde{Z})^{\mathrm{biZar}}.
\end{equation*}

\subsection{Universal abelian variety}
Let $\pi \colon \mathfrak{A}_g \rightarrow \mathbb{A}_g$ be the universal abelian variety.

The uniformization of $\mathfrak{A}_g$, in the category of complex
spaces, is given by theta functions
\begin{equation}\label{EqUnifUniversalAbVar}
u \colon \IC^g \times  \mathfrak{H}_g \rightarrow \mathfrak{A}_g^{\mathrm{an}}.
\end{equation}

As for $\mathfrak{H}_g$, the complex space $\IC^g \times  \mathfrak{H}_g$ is naturally an open subset of the analytification of the complex algebraic variety $\IC^g \times M_{g \times g}(\mathbb{C})$. Thus we endow $\IC^g \times  \mathfrak{H}_g$ with the following algebraic structure, and therefore define a bi-algebraic structure on $\mathfrak{A}_g$.
\begin{definition}
\begin{enumerate}
\item[(i)] A subset $\tilde{Y} \subseteq \IC^g \times  \mathfrak{H}_g$
  is said to be irreducible algebraic if $\tilde{Y}$ is a complex
  analytic irreducible component of $W \cap (\IC^g \times
  \mathfrak{H}_g)$, where $W$ is an algebraic subset of $\IC^g
  \times M_{g \times g}(\mathbb{C})$.
\item[(ii)] An irreducible subvariety $Y$ of $\mathfrak{A}_g$ is said to be bi-algebraic if $Y = u(\tilde{Y})$ for some irreducible algebraic subset $\tilde{Y}$ of $\IC^g \times  \mathfrak{H}_g$.
\end{enumerate}
\end{definition}

We refer to \cite[Proposition~5.3]{GaoBettiRank} for a geometric description of bi-algebraic subvarieties of $\mathfrak{A}_g$. 
As for the moduli space, the following notation will be used. Let
$\tilde{Z}$ be a complex analytic irreducible subset of $\IC^g \times
\mathfrak{H}_g$. We use $\tilde{Z}^{\mathrm{Zar}}$ to denote the
smallest algebraic subset of $\IC^g \times \mathfrak{H}_g$ that
contains $\tilde{Z}$
and use $u(\tilde{Z})^{\mathrm{biZar}}$ to
denote the smallest bi-algebraic subvariety of $\mathfrak{A}_g$ that
contains $u(\tilde{Z})$. We have the following relation 
\begin{equation}
  u(\tilde{Z})^{\mathrm{Zar}} \subseteq u(\tilde{Z})^{\mathrm{biZar}}.
\end{equation}

The following lemma summarizes the bi-algebraic structures discussed above.
\begin{lemma}
We have the following  commutative diagram:
\begin{equation*}
\xymatrix{
 \IC^g \times \mathfrak{H}_g  \ar[r]^-{u} \ar[d]_{\tilde{\pi}} & \mathfrak{A}_g^{\mathrm{an}} \ar[d]^{\pi} \\ \mathfrak{H}_g \ar[r]^-{u_B} & \mathbb{A}_g^{\mathrm{an}},
}
\end{equation*}
where $\tilde{\pi}$ is the natural projection (hence is the
restriction of an algebraic map), and both $u$ and $u_B$ are uniformizations in the category of complex spaces.
\end{lemma}

\subsection{Functional transcendence for $\mathfrak{A}_g$}
The following theorem, known as the \textit{weak mixed Ax--Schanuel
  theorem for $\mathfrak{A}_g$}, was
proved by the first-named author.
\begin{theorem}[{\!\!\cite[Theorem~3.5]{GaoMixedAS}}]\label{ThmAS}
Let $\tilde{Z}$ be an irreducible complex analytic 
subset of $\IC^g \times\mathfrak{H}_g$. Then we have
\[
\dim \tilde{Z}^{\mathrm{Zar}} + \dim u(\tilde{Z})^{\mathrm{Zar}} \ge \dim \tilde{Z} + \dim u(\tilde{Z})^{\mathrm{biZar}}.
\]
\end{theorem}

\subsection{The universal uniformized Betti map}
The following Betti map has been proved to be important in several Diophantine problems.

Consider the real-algebraic isomorphism
\begin{equation}\label{EqComplexStrOfX2g}
\begin{array}{cccc}
 \mathbb{R}^g \times \mathbb{R}^g \times \mathfrak{H}_g & \xrightarrow{\sim} & \mathbb{C}^g \times \mathfrak{H}_g, \\
 (a,b,Z) & \mapsto & (Da+Zb, Z)
\end{array}.
\end{equation}
The \textit{universal uniformized Betti map} is the semi-algebraic map
\begin{equation}\label{EqUnivBettiMap}
\tilde{b} \colon \IC^{g} \times \mathfrak{H}_g \rightarrow \IR^{2g},
\end{equation}
which is the inverse of \eqref{EqComplexStrOfX2g} composed with the natural projection $ \mathbb{R}^g \times \mathbb{R}^g \times \mathfrak{H}_g \rightarrow \IR^{2g}$.


%% file: deglocus.tex
\subsection{Degeneracy loci}
An important notion in our approach to prove the Relative
Manin--Mumford Conjecture is the \textit{$t$-th degeneracy locus}
introduced by the first-named author in \cite[Definition~1.6]{GaoBettiRank}.
In this paper, we need the cases $t = 0$ and $t = 1$.

Let $g\ge 1$. We consider $\mathfrak{A}_g$ as a smooth, irreducible,
quasi-projective variety defined over $\IC$. Recall that $\pi\colon
\mathfrak{A}_g\rightarrow\IA_g$ is the structural morphism.

\begin{definition}
Let $X$ be an irreducible subvariety of $\mathfrak{A}_g$. Let $t \in \IZ$. The $t$-th degeneracy locus, denoted by $X^{\mathrm{deg}}(t)$, is the union
\begin{equation}\label{EqDefnDegLocus}
X^{\mathrm{deg}}(t) = \bigcup_{\substack{Y \subseteq X\text{ irreducible,}~ \dim Y > 0 \\ \dim Y^{\mathrm{biZar}} - \dim \pi(Y)^{\mathrm{biZar}} < \dim Y + t}} Y.
\end{equation}
\end{definition}
We refer to $\mathsection$\ref{SectionBiAlg} for the notations
$Y^{\mathrm{biZar}}$ and $\pi(Y)^{\mathrm{biZar}}$. However, we will
not need this until $\mathsection$\ref{SectionAppAS}. Notice that
\eqref{EqDefnDegLocus} is not precisely \cite[Definition~1.6]{GaoBettiRank}, but these two definitions are equivalent for $X
\subseteq \mathfrak{A}_g$ by \cite[Corollary~5.4]{GaoBettiRank}. It is proved in \cite{GaoBettiRank} that $X$ is non-degenerate, \textit{i.e.} the generic Betti rank of $X$ is $2\dim X$, if and only if $X^{\mathrm{deg}}(0) \not= X$.

An immediate consequence of the definition is $X^{\mathrm{deg}}(0) \subseteq X^{\mathrm{deg}}(1)$.

The degeneracy locus is a possibly infinite union of Zariski closed
subsets. Nevertheless we have the following theorem.
\begin{theorem}[{\!\!\cite[Theorem~1.8]{GaoBettiRank}}]
   \label{ThmDegLocusZarClosed}
   The degeneracy locus $X^{\mathrm{deg}}(t)$ is  Zariski closed
   in $X$ for all $t$.
\end{theorem} 

In the current paper, we apply this theorem in
$\mathsection$\ref{SectionConclusion}. Moreover, if $X$ is defined
over $\IQbar$, then $X^{\mathrm{deg}}(t)$ is defined over $\IQbar$ by
\cite[Theorem~4.2.4]{GaoHDR}; however we will not use this fact in the
current work. 

\subsection{The $0$-th degeneracy locus and the height inequality}
Let $g\ge 1$. In this and the next section we consider $\mathfrak{A}_g$ as a smooth,
irreducible, quasi-projective variety defined over $\IQbar$.

The $0$-th degeneracy locus  is closely related
to the following height inequality proved by both authors of the
current paper and Dimitrov.

Below we fix a height function $h\colon S(\IQbar)\rightarrow [0,\infty)$
and a fiberwise canonical height $\hat h\colon
\mathfrak{A}_g(\IQbar)\rightarrow[0,\infty)$
as in \cite{DGHUnifML}. 

\begin{theorem}[{\!\!\cite[Theorem~1.6 and Theorem~B.1]{DGHUnifML}}]\label{ThmHtInequality}
  Let $X$ be an irreducible subvariety of $\mathfrak{A}_g$ defined over $\IQbar$. Assume $X_\IC^{\mathrm{deg}}(0)$ is not Zariski dense in $X$. Then there exist a constant $c>0$ and a Zariski open dense subset $U$ of $X$ such that
  \begin{equation}\label{EqHtInequality}
    h(\pi(x)) \le c ( \hat{h}(x) + 1) \quad \text{for all }x \in U(\IQbar).
  \end{equation} 
\end{theorem}

Indeed, the assumption of \cite[Theorem~B.1]{DGHUnifML} is that $X$ is
\textit{non-degenerate} as defined by \cite[Definition~B.4]{DGHUnifML}.
If $X$ fails to be non-degenerate, then 
the mixed Ax--Schanuel Theorem for $\mathfrak{A}_g$
by the first-named author~\cite{GaoMixedAS} implies that
$X_\IC^{\mathrm{deg}}(0)$
contains a non-empty  open subset of $X^{\mathrm{an}}$, 
see \cite[Theorem~1.7]{GaoBettiRank} or 
\cite[Theorem~4.3.1]{GaoHDR}.

\subsection{The $1$-st degeneracy locus and Relative Manin--Mumford}
The $1$-st degeneracy locus 
is closely related to the Relative Manin--Mumford Conjecture. In
 $\mathsection$\ref{SectionConclusion} we will deduce Theorem~\ref{MainThm}
in the case $L=\IQbar$ from the following theorem.

\begin{theorem}\label{TheoremFirstDegLocusIsX}
Let $X$ be an irreducible subvariety of $\mathfrak{A}_g$ defined over
$\IQbar$ with $\dim X\ge 1$. If $X(\IQbar) \cap (\mathfrak{A}_g)_{\mathrm{tor}}$ is Zariski dense in $X$, then $X_\IC^{\mathrm{deg}}(1)$ is Zariski dense in $X_\IC$.
\end{theorem}
A large part of this paper
($\mathsection$\ref{SectionLGO}--\ref{SectionAppAS}) is devoted to the
proof of Theorem~\ref{TheoremFirstDegLocusIsX}.

Our proof of Theorem~\ref{TheoremFirstDegLocusIsX} will be arranged as
follows. We divide into two cases: either $X_\IC^{\mathrm{deg}}(0)$ is
Zariski dense in $X$, or $X_\IC^{\mathrm{deg}}(0)$ is not Zariski
dense in $X$. In the first case, the conclusion of
Theorem~\ref{TheoremFirstDegLocusIsX} holds true because
$X_\IC^{\mathrm{deg}}(0) \subseteq X_\IC^{\mathrm{deg}}(1)$. In the second
case, we follow the Pila--Zannier method. More precisely, we will
invoke the height inequality \eqref{EqHtInequality} and a result of 
R\'{e}mond \cite[Proposition~2.9]{RemondGaloisBound} to show that the Galois orbit
of each point in $X(\IQbar) \cap (\mathfrak{A}_g)_{\mathrm{tor}}$ is large,
and then apply the semi-rational variant of the Pila--Wilkie counting
theorem \cite[Corollary~7.2]{HabeggerPilaENS}.
One may also use an older result of David
\cite{DavidMinHaut} to deduce largeness of the Galois orbit, as
explained in the appendix.
We will thus
produce a curve whose projection to the Betti fiber is a
semi-algebraic curve. Finally, by the mixed Ax--Schanuel theorem for
$\mathfrak{A}_g$ (Theorem~\ref{ThmAS}) such a curve will be seen to contribute to
$X_\IC^{\mathrm{deg}}(1)$.



%% file: Galorbit.tex

The goal of this section is to prove the following result, claiming
that the Galois orbits of torsion points are \textit{large}.

Let $g\ge 1$. We consider $\mathfrak{A}_g$ as a smooth, irreducible,
quasi-projective variety defined over $\IQbar$.
Let $X$ be an irreducible subvariety of $\mathfrak{A}_g$ defined over
$\IQbar$. Suppose $K\subset\IQbar$ is a number field over which $X$
and $\mathfrak{A}_g$ are defined. For each $x \in X \cap
(\mathfrak{A}_g)_{\mathrm{tor}}$, let $\ord{x}$ denote the order of
$x$.

\begin{proposition}\label{PropLGO}
  Assume that $X_\IC^{\mathrm{deg}}(0)$ is not Zariski dense in $X$. Then there exist constants $c>0$, $\delta>0$, and a Zariski open dense subset $U$ of $X$ such that
  \begin{equation}
    \#\mathrm{Gal}(\IQbar/K) x \ge c \cdot \ord{x}^{\delta} \quad \text{ for all }x \in U(\IQbar) \cap (\mathfrak{A}_g)_{\mathrm{tor}}.
  \end{equation}
\end{proposition}

The key ingredients of the proof are the height inequality
\cite[Theorems 1.6 and B.1]{DGHUnifML} and a \textit{quantitative} version of Masser's work on  the Galois orbit of torsion points on abelian varieties. In the proof presented here, we use R\'{e}mond's \cite[Proposition~2.9]{RemondGaloisBound}, which
is based on
earlier work of 
Gaudron--R\'{e}mond \cite{GR:Isogeny} itself related to isogeny
estimates of  Masser--W\"{u}stholz~\cite{MW:abelianisog}.
Alternatively, one can use \cite[Corollaire~1.7]{GRGaloisBound}.
In the Appendix, we show how an older result of David \cite{DavidMinHaut} suffices for our purpose. 



\begin{proof}
Let $U$ be the Zariski open dense subset of $X$ obtained from
Theorem~\ref{ThmHtInequality}, and let $c = c(X) >0$ be the constant
from the same theorem. Then
\begin{equation}\label{Eq}
  h(\pi(x)) \le c \qquad \text{for all }x \in U(\IQbar) \cap (\mathfrak{A}_g)_{\mathrm{tor}}
\end{equation}
as the N\'eron--Tate height of a torsion point is $0$.


For $x\in U(\IQbar)\cap (\mathfrak{A}_g)_{\mathrm{tor}}$
the abelian variety $A=\mathfrak{A}_{g,x}$
is  defined over $k=K(x)$ with $K$ as in the beginning of
this section. 

Let $h_{\mathrm{Fal}}(A)$ denote the \textit{stable Faltings height} of
$A$. The additive normalization in $h_{\mathrm{Fal}}$ plays no role in
the current work. 
Because of \eqref{Eq} and
by the fundamental work of Faltings~\cite[$\mathsection 3$ including
 the proof of Lemma~3]{Faltings:ES} (see also
\cite[the remarks below Proposition~V.4.4 and Proposition~V.4.5]{FaltingsChai}), we have $h_{\mathrm{Fal}}(A) \le c'$ for some constant $c' = c'(X)>0$.

Define
\[
\kappa(X) = ((14g)^{64g^2} [k:\IQ] \max\{1,c', \log [k:\IQ]\}^2)^{1024 g^3}.
\]
Then by R\'{e}mond   \cite[Proposition~2.9]{RemondGaloisBound}, we have $\ord{x} \le \kappa(X)^{4g+1}$ because $h_{\mathrm{Fal}}(A) \le c'$. Hence we are done.
\end{proof}


%% file: PointCount.tex
Let $g\ge 1$ be an integer. In this and the next section we consider
$\mathfrak{A}_g$ as a smooth, irreducible, quasi-projective variety
defined over $\IQbar$.
Let $X$ be an irreducible subvariety of $\mathfrak{A}_g$, as usual defined over
$\IQbar$. Let $K$ be a number field such that both $\mathfrak{A}_g$ and $X$
are defined over $K$. Recall the uniformization $u \colon \IC^g \times \mathfrak{H}_g
\rightarrow \mathfrak{A}_g^{\mathrm{an}}$ from
\eqref{EqUnifUniversalAbVar}.
Let $\tilde{b} \colon \IC^g \times \mathfrak{H}_g \rightarrow \IR^{2g}$ be the universal uniformized Betti map from \eqref{EqUnivBettiMap}; it is a semi-algebraic map.

The goal of this section is to prove the following result. Let $\IRanexp$ denote the  structure generated by the unrestricted
exponential function on the real numbers and the restriction of all
real analytic functions to a hypercube. Then $\IRanexp$ is o-minimal by the
Theorem of van den Dries and Miller \cite{DriesMiller:94} and earlier work of Wilkie \cite{Wilkie:96}.
Throughout this section, the word \textit{definable} will mean definable in $\IRanexp$.

\begin{proposition}\label{PropositionSemiAlgCounting}
  Assume $\dim X\ge 1$ and
  that $X_\IC^{\mathrm{deg}}(0)$ is not Zariski dense in $X$. Then there exists a Zariski open
  dense subset $U$ of $X$ and $B\ge 1$ with the following property. For each $x \in
  U(\IQbar) \cap (\mathfrak{A}_g)_{\mathrm{tor}}$ with
  $\mathrm{ord}(x)\ge B$
  there exist $\tilde\gamma_x \colon
  [0,1]\rightarrow\IC^g\times \mathfrak{H}_g$ continuous and definable with
  the following properties. 
  \begin{enumerate}
  \item[(i)] The map $\tilde\gamma_x$ is non-constant,
    $\tilde\gamma_x|_{(0,1)}$ is real analytic,  and $u(\tilde\gamma_x([0,1]))\subset
    X^{\mathrm{an}}$, 
  \item[(ii)] $u(\tilde\gamma_x(0)) \in \mathrm{Gal}(\IQbar/K)x$, and
  \item[(iii)] $\tilde{b} \circ \tilde{\gamma}_x\colon
    [0,1]\rightarrow\IR^{2g}$ is semi-algebraic. 
  \end{enumerate}
\end{proposition}

\begin{proof} Let $U$ be the Zariski open dense subset of $X$ from Proposition~\ref{PropLGO}.



  
Let $\tau \colon \IR^{2g} \times \mathfrak{H}_g \rightarrow \IC^g \times \mathfrak{H}_g$ be the semi-algebraic isomorphism given by \eqref{EqComplexStrOfX2g}. We have the following commutative diagram.
\[
\xymatrix{
\IC^g \times \mathfrak{H}_g \ar[rd]_-{\tilde{b}} & \IR^{2g} \times \mathfrak{H}_g \ar[l]_-{\tau} \ar[d] \\
 & \IR^{2g}
}
\]

By work of Peterzil--Starchenko
\cite{PeterzilStarchenkoDefinabilityTheta}, there exists a
semi-algebraic fundamental subset $\mathfrak{F}_0 \subseteq
\mathfrak{H}_g$ with the following property: for each integer $M > 0$,
if we set $\mathfrak{F} = \tau([-M,M]^{2g} \times \mathfrak{F}_0)
\subseteq \tau(\IR^{2g} \times \mathfrak{H}_g) = \IC^g \times
\mathfrak{H}_g$, then $u|_{\mathfrak{F}} \colon \mathfrak{F}
\rightarrow \mathfrak{A}_g^{\mathrm{an}}$ is  definable in the
o-minimal structure $\IRanexp$;
observe that we may assume that $\mathfrak{A}_g$ is embedded in some
projective space. 
Note that $\mathfrak{F}$ is a semi-algebraic subset of $\IC^g \times \mathfrak{H}_g$ because both $\mathfrak{F}_0$ and $\tau$ are semi-algebraic.

Fix an integer $M \ge 1$ large enough such that $u(\mathfrak{F}) = \mathfrak{A}_g^{\mathrm{an}}$.


Consider $\IR^{2g} \times \mathfrak{F} \subseteq \IR^{2g} \times \IC^g \times \mathfrak{H}_g$. Define
\begin{equation}\label{EqDefnDefFamily}
\hat{X} = \{(r, \tilde{x}) \in \IR^{2g} \times \mathfrak{F} : \tilde{x} \in (u |_{\mathfrak{F}})^{-1}(X(\IC)), ~ r = \tilde{b}(\tilde{x})\}.
\end{equation}
As both $u|_{\mathfrak{F}}$ and $\tilde{b}$ are definable maps, the set $\hat{X}$ is a definable subset of $\IR^{2g} \times \mathfrak{F}$. By choice of $\mathfrak{F}$, the following holds true:  for any $(r,\tilde{x}) \in \hat{X}$, we have $r \in [-M,M]^{2g}$.

The height $H(p/q)$ of a rational number $p/q$ with coprime integers
$p,q$ and $q\ge 1$ is $\max\{|p|,q\}$. The height $H(r_1,\ldots,r_n)$
with $r_1,\ldots,r_n\in \IQ$ is $\max_i H(r_i)$. 

For each $T \ge 1$, set $\hat{X}(T) = \{(r,\tilde{x}) \in \hat{X} : r \in \IQ^{2g},~ H(r) \le T\}$. 

Now let $x \in X(\IQbar) \cap (\mathfrak{A}_g)_{\mathrm{tor}}$ and let
$T=\ord{x}$ be the order of $x$.

Choose $\tilde{x} \in (u|_{\mathfrak{F}})^{-1}(x)$ and set $r = \tilde{b}(\tilde{x})$. Then $r \in \IQ^{2g}$ and $T\cdot r \in \IZ^{2g}$. As $r \in [-M,M]^{2g}$ by choice of $\mathfrak{F}$, we have $H(r) \le MT$. So $(r,\tilde{x}) \in \hat{X}(MT)$.

Thus we can invoke Proposition~\ref{PropLGO} to obtain a subset $\Sigma \subseteq \hat{X}(MT)$ as follows. Recall that $X$ is defined over a number field $K$, and consider $\mathrm{Gal}(\IQbar/K)x \subseteq X(\IQbar)$. Let $c$ and $\delta$ be the constants from Proposition~\ref{PropLGO}, then we have $\#\mathrm{Gal}(\IQbar/K)x \ge c T^{\delta}$. 
Each point $x' \in \mathrm{Gal}(\IQbar/K)x$ is again in $X(\IQbar)
\cap (\mathfrak{A}_g)_{\mathrm{tor}}$ of order $T$. Hence it gives rise to a point $(r',\tilde{x}') \in \hat{X}(MT)$. Let $\Sigma$ be the set of all such points in $\hat{X}(MT)$. Then for the natural projection $\pi_2 \colon \IR^{2g}\times \mathfrak{F} \rightarrow \mathfrak{F}$, we have
\begin{equation}
 \#\pi_2(\Sigma) \ge cT^{\delta}. 
\end{equation}
We may fix $B$ large in terms of $c,M,$ and $c'$, the constant in
the semi-rational variant of the Pila--Wilkie
counting theorem \cite[Cor.7.2]{HabeggerPilaENS} with $\epsilon =
\delta/2$, to the following effect.
If  $T = \ord{x}\ge B$, then
\begin{equation}
 \#\pi_2(\Sigma)  > c'(MT)^{\delta/2}.
\end{equation}



We apply \cite[Cor.7.2]{HabeggerPilaENS} to $\hat X$, taken as a
single-fiber family, $\Sigma \subseteq \hat{X}(MT)$, and $\epsilon =
\delta/2$. We obtain a continuous and definable map $\beta \colon
[0,1] \rightarrow \hat{X}$ such that
\begin{enumerate}
\item[(a)] the composition $[0,1]\xrightarrow{\beta} \hat{X} \subseteq \IR^{2g} \times \mathfrak{F} \xrightarrow{\pi_1} \IR^{2g}$ is semi-algebraic;
\item[(b)] the composition $\pi_2 \circ \beta \colon [0,1]\rightarrow \hat{X} \rightarrow \mathfrak{F}$ is non-constant;
\item[(c)] $\pi_2(\beta(0)) \in \pi_2(\Sigma)$;
\item[(d)] $\beta|_{(0,1)}$ is real analytic. 
\end{enumerate}
For the final property we require that $\IRanexp$ has analytic
cell decomposition, this follows from the
work of van den Dries and Miller \cite{DriesMiller:94}. 

Let $\tilde\gamma_x = \pi_2\circ\beta\colon [0,1]\rightarrow
\IC^g\times \mathfrak{H}_g$. 
We will show that this $\tilde\gamma_x$ is what
we desire.


First, $u(\tilde\gamma_x([0,1])) = u(\pi_2\circ \beta([0,1])) \subseteq
u(\pi_2(\hat{X}))\subset X(\IC)$. 
 Moreover, $\tilde\gamma_x$ is non-constant by property (b) above
and its restriction to $(0,1)$ is real analytic by (d). So part (i) follows.

By property (c), we have $u(\tilde\gamma_x(0)) = (u\circ \pi_2\circ\beta)(0) \in u(\pi_2(\Sigma)) \subseteq \mathrm{Gal}(\IQbar/K)x$. Hence (ii) is established.

It remains to establish property (iii). Notice that
$\pi_1|_{\hat X} = \tilde b \circ \pi_2|_{\hat X}$ by
(\ref{EqDefnDefFamily}).
As $\beta$ takes values in $\hat X$ we find
$\pi_1\circ\beta = \tilde b \circ\pi_2\circ\beta = \tilde b
\circ\tilde \gamma_x$. So $\tilde b \circ\tilde\gamma_x$ is
semi-algebraic by (a). 
\end{proof}


%% file: AppAS.tex
Let $g\ge 1$. In this and the next section we consider
$\mathfrak{A}_g$ as a smooth, irreducible, quasi-projective variety
defined over $\IQbar$.

\begin{proposition}\label{PropXIsDeg1}
  Let $X$ be an irreducible subvariety of
  $\mathfrak{A}_g$ and assume that
  $\dim X\ge 1$, that $X_\IC^{\mathrm{deg}}(0)$ is not Zariski dense
  in $X$, and that $X(\IQbar)\cap
  (\mathfrak{A}_g)_{\mathrm{tor}}$ is Zariski dense in $X$.
  Then $X_\IC^{\mathrm{deg}}(1)$ is Zariski dense in $X_\IC$.
\end{proposition}

\begin{proof}
  Let $K$ be a number field over which $\mathfrak{A}_g$ and $X$ are
  defined.

  Let $U\subset X$ and $B\ge 1$ be 
  from Proposition~\ref{PropositionSemiAlgCounting}. We may apply
  Proposition~\ref{PropositionSemiAlgCounting} to elements of $\Sigma=
  \{ x\in U(\IQbar) \cap (\mathfrak{A}_g)_{\mathrm{tors}} :
  \mathrm{ord}(x) \ge B\}$. By replacing $U$ by a Zariski open and dense
  subset we may assume that $\Sigma$ is stable under the action of
  $\mathrm{Gal}(\IQbar/K)$.

  Suppose that $\Sigma$ is not Zariski dense in $X$. Then by hypothesis $X$ contains a Zariski
  dense set of torsion points of order  $<B$. Then $X \subset
  \mathrm{ker}[n]$ for some positive integer $n< B$. As $\dim X >0$ we
  find that $X_\IC$ appears in the union (\ref{EqDefnDegLocus}) for $t=0$. In particular,  
  $X_\IC^{\mathrm{deg}}(0)=X$ which contradicts our hypothesis.
  So $\Sigma$ is Zariski dense in $X$. 


  Take $x \in \Sigma$ and 
  let $\tilde\gamma_x$ be as in Proposition~\ref{PropositionSemiAlgCounting}.

  Let $\tilde{C} = \tilde\gamma_x([0,1]) \subset\IC^g\times\mathfrak{H}_g$. This is a connected and
  definable set as $\tilde\gamma_x$ is continuous and definable;
  recall that definable means definable in $\IRanexp$.
  Moreover, $\dim\tilde{C} \ge 1$ since $\tilde\gamma_x$ is non-constant.

  We let $\tilde Z$ denote the intersection of all complex analytic subsets
  of $\IC^g\times \mathfrak{H}_g$ containing $\tilde C$. Then $\tilde Z$ is
  itself a complex analytic subset of $\IC^g\times\mathfrak{H}_g$.
  Moreover, $\tilde Z$ is irreducible as $\tilde\gamma_x|_{(0,1)}$ is real
  analytic. Finally, $\dim \tilde Z\ge 1$.

  We define $Y\subset\mathfrak{A}_{g,\IC}$
  to be the Zariski closure $u(\tilde Z)^{\mathrm{Zar}}$. Then $Y$ is
  irreducible and $\dim Y\ge 1$. 

  Next we apply the first-named author's mixed Ax--Schanuel Theorem for the
  universal family to obtain the following lemma. 
  \begin{lemma}
    \label{LemmaZarClosureinXdeg1}
    We have $Y \subseteq X_\IC^{\mathrm{deg}}(1)$.
  \end{lemma}

  Let us finish the proof of Proposition~\ref{PropXIsDeg1} assuming
  Lemma~\ref{LemmaZarClosureinXdeg1}. By property (ii) of
  Proposition~\ref{PropositionSemiAlgCounting}, we have $x' \in
  u(\tilde{C})$  for some $x' \in \mathrm{Gal}(\IQbar/K)x$. Thus $x' \in
  Y(\IC)$. Assuming Lemma~\ref{LemmaZarClosureinXdeg1}, we then have $x'
  \in X_\IC^{\mathrm{deg}}(1)$.

  To summarize, the orbit under $\mathrm{Gal}(\IQbar/K)$ of any given
  element in $\Sigma$ meets $X_\IC^{\mathrm{deg}}(1)$. Recall that
  $\mathrm{Gal}(\IQbar/K)$ acts on 
  $\Sigma$.
  So there is a subset $\Sigma'\subset \Sigma$ contained
  completely in $X_\IC^{\mathrm{deg}}(1)$ such that for all $x\in \Sigma$
  there is $\sigma\in \mathrm{Gal}(\IQbar/K)$ with $\sigma(x)\in
  \Sigma'$.
  The Zariski closure of $\Sigma'$ in $X_\IC$ is defined over $\IQbar$.
  So it is stable under a finite index subgroup of
  $\mathrm{Gal}(\IQbar/K)$.
  As $\Sigma$ is Zariski dense in $X$ we conclude that $\Sigma'$ is
  Zariski dense in $X$ as well.

  As $\Sigma' \subset X_\IC^{\mathrm{deg}}(1)$ we conclude that
  $X_\IC^{\mathrm{deg}}(1)$ is Zariski dense in $X_\IC$, as desired. 
\end{proof}

\begin{proof}[Proof of Lemma~\ref{LemmaZarClosureinXdeg1}]
  We systematically work with varieties defined over $\IC$ and  also identify
  varieties with their set of complex points.
  To ease notation we write $X$ for the base change $X_\IC$.

  We apply the weak Ax--Schanuel theorem for $\mathfrak{A}_g$,
  Theorem~\ref{ThmAS}, to $\tilde Z$ and get
  \begin{equation*}
    \dim {\tilde Z}^{\mathrm{Zar}} + \dim u(\tilde Z)^{\mathrm{Zar}} \ge \dim \tilde Z + \dim u(\tilde Z)^{\mathrm{biZar}}.
  \end{equation*}
  Recall $\dim \tilde Z\ge 1$ and  $Y=u(\tilde
  Z)^{\mathrm{Zar}}\subset u(\tilde Z)^{\mathrm{biZar}}$. So
  $Y^{\mathrm{biZar}}\subset  u(\tilde Z)^{\mathrm{biZar}}$
  and hence
  \begin{equation}
    \label{eq:masineq}
    \dim {\tilde Z}^{\mathrm{Zar}} + \dim Y \ge 1 + \dim Y^{\mathrm{biZar}}.
  \end{equation}  

  Consider the following commutative (possibly non-Cartesian) diagram involving uniformizations and projections to the bases:
  \begin{equation*}
    \xymatrix{
      \IR^{2g} & \IC^g \times \mathfrak{H}_g \ar[l]_-{\tilde{b}} \ar[r]^-{u} \ar[d]_{\tilde{\pi}} & \mathfrak{A}_g^{\mathrm{an}} \ar[d]^{\pi} \\
      & \mathfrak{H}_g \ar[r]^-{u_B} & \mathbb{A}_g^{\mathrm{an}}.
    }
  \end{equation*}
  Recall that  $\tilde b$
  is the universal uniformized Betti map defined in
  \eqref{EqUnivBettiMap}.

  We have
  \begin{equation*}
    u_B(\tilde \pi(\tilde C)) = \pi(u(\tilde C)) \subset
    \pi(u(\tilde Z)) \subset \pi(u(\tilde Z)^{\mathrm{Zar}})=\pi(Y) \subset \pi(Y)^{\mathrm{biZar}}.
  \end{equation*}
  So $\tilde\pi\circ\tilde \gamma_x$ takes values in
  $u_B^{-1}(\pi(Y)^{\mathrm{biZar}})$. As
  $\tilde\pi\circ\tilde\gamma_x|_{(0,1)}$ is real analytic and
  continuous at the boundary, the values
  of $\tilde\pi\circ\tilde \gamma_x$ lie in an complex analytic
  irreducible component $\tilde Y_0$ of   $u_B^{-1}(\pi(Y)^{\mathrm{biZar}})$. By
  the definition of bi-algebraic subsets of $\IA_g$ there is an
  algebraic subset $\tilde Y$ of $\mathrm{Mat}_{g\times g}(\IC)$ such that
  $\tilde Y_0$ is a complex analytic irreducible component of
  $\tilde Y \cap \mathfrak{H}_g$.  We may assume that $\tilde Y$ is
  irreducible.
  In particular,
  \begin{equation}
    \label{eq:dimYequality}
   \dim \tilde Y = \dim \tilde Y_0 =
  \dim \pi(Y)^{\mathrm{biZar}}. 
  \end{equation}

  Recall that $\tilde\pi(\tilde C)=\tilde\pi(\tilde\gamma_x([0,1]))
  \subset \tilde Y_0$. So
  $$
  \tilde C \subset\IC^g \times \tilde Y_0 \subset\IC^g\times \tilde Y.
  $$

  On the other hand, $\tilde b\circ\tilde\gamma_x$ is semi-algebraic
  by Proposition~\ref{PropositionSemiAlgCounting}(iii). In other
  words, the Betti coordinates along $\tilde\gamma_x$ lie in a
  real semi-algebraic curve in $\IR^{2g}$. Thus there is a complex algebraic curve
  $A\subset\IC^{2g}$ with $\tilde C\subset \{((\tau,1_g)z,\tau) : z\in
  A\text{ and }\tau\in\mathfrak{H}_g\}$ (we take points of $A$ as
  column vectors).
  We conclude $\tilde C\subset E$ with $E$ the Zariski closure of 
  \begin{equation*}
    (\IC^g\times \tilde Y)\cap \{((\tau,1_g)z,\tau) : z\in
    A\text{ and }\tau\in\mathrm{Mat}_{g\times g}(\IC)\}.
  \end{equation*}
  
  Recall that $\tilde Z$ is the smallest complex analytic subset of
  $\IC^g\times \mathfrak{H}_g$ containing
  $\tilde C$. So it is contained in $E$. Moreover,  $E$ is complex
  algebraic, hence
    ${\tilde Z}^{\mathrm{Zar}} \subset E$.
  Finally, all fibers of the restricted projection $E\rightarrow
  \mathrm{Mat}_{g\times g}(\IC)$ have dimension at most $1$.
  Therefore, $\dim E\le 1+\dim \tilde Y$ and hence
  $\dim {\tilde Z}^{\mathrm{Zar}}\le \dim E \le \dim 1 + \dim \tilde Y = 1 +\dim
  \pi(Y)^{\mathrm{biZar}}$ by (\ref{eq:dimYequality}). 

  We apply (\ref{eq:masineq}) and conclude
  $\dim Y \ge \dim Y^{\mathrm{biZar}} - \dim \pi(Y)^{\mathrm{biZar}}$.
  As $\dim Y\ge 1$ we find $Y\subset X^{\mathrm{deg}}(1)$ by
  (\ref{EqDefnDegLocus}).
  This concludes the proof of Lemma~\ref{LemmaZarClosureinXdeg1}.
\end{proof}

\subsection{Proof of Theorem~\ref{TheoremFirstDegLocusIsX}}
Let $X$ be an irreducible subvariety of $\mathfrak{A}_g$ defined over $\IQbar$ and $\dim X \ge 1$. Assume that $X(\IQbar) \cap (\mathfrak{A}_g)_{\mathrm{tor}}$ is Zariski dense in $X$.

We are in two cases: either $X_\IC^{\mathrm{deg}}(0)$ is or is not Zariski dense
 in $X$, for $X_\IC^{\mathrm{deg}}(0)$ defined by
\eqref{EqDefnDegLocus} (with $t=0$).

If $X_\IC^{\mathrm{deg}}(0)$ is Zariski dense in $X$, so is $X_\IC^{\mathrm{deg}}(1)$ because $X_\IC^{\mathrm{deg}}(1) \supseteq X_\IC^{\mathrm{deg}}(0)$ by definition of $X_\IC^{\mathrm{deg}}(t)$.

If $X_\IC^{\mathrm{deg}}(0)$ is not Zariski dense in $X$, then we can
invoke Proposition~\ref{PropXIsDeg1} to conclude that
$X_\IC^{\mathrm{deg}}(1)$ is Zariski dense in $X_\IC$.
Hence we are done.\qed


%% file: RMM.tex
Now we are ready to prove  our main theorem, Theorem~\ref{MainThm}, when the base field is $L=\IQbar$,
using Theorem~\ref{TheoremFirstDegLocusIsX}.
The deduction, which uses the criterion for $X_{\IC}^{\mathrm{deg}}(1) = X$, was given by \cite[Proposition 1.10]{GaoBettiRank} and with more details 
by \cite[Corollary 6.3]{GH:nondeg} for the universal family of
principally polarized abelian varieties.
To make the current paper more self-contained, we represent and simplify this proof.
\begin{proof}[Proof of Theorem~\ref{MainThm} over $\IQbar$]
Let us first reduce to the case where $\cA \rightarrow S$ satisfies the following assumption:
\begin{center}
  {\tt (Hyp):} $S$ is a regular locally closed subvariety of $\mathbb{A}_g$ defined over $\IQbar$ and $\cA = \mathfrak{A}_g \times_{\mathbb{A}_g} S$.
\end{center}
Indeed, let $X \subseteq \cA \rightarrow S$ be defined over $\IQbar$ satisfying the assumptions of Theorem~\ref{MainThm}, \textit{i.e.}, 
$\mathbb{Z}X$ is Zariski dense in $\cA$ and $X(\IQbar) \cap \cA_{\mathrm{tor}}$ is Zariski dense in $X$. Take a relatively ample line bundle on $\cA \rightarrow S$. By \cite[$\mathsection$2.1]{GenestierNgo}, it induces a polarization of type $D = (d_1,\ldots,d_g)$ on $\cA\rightarrow S$ with $d_1|\cdots|d_g$. Take $\ell \gg 1$ large enough with $(\ell, d_g) = 1$. Then there exists a finite \'{e}tale morphism $\rho \colon S' \rightarrow S$ such that $\cA' := \cA\times_S S' \rightarrow S'$ has level-$\ell$-structure. Write $\rho_{\cA} \colon \cA' \rightarrow \cA$ for the natural projection. Let $X'$ be an irreducible component of $\rho_{\cA}^{-1}(X)$; then $\dim X' = \dim X$. It is not hard to show that 
$\mathbb{Z}X'$ is Zariski dense in $\cA'$ and $X'(\IQbar) \cap \cA'_{\mathrm{tor}}$ is Zariski dense in $X'$. 

Next we have a Cartesian diagram
\[
\xymatrix{
\cA' \ar[r]^-{\iota'} \ar[d] \pullbackcorner & \mathfrak{A}_g \ar[d] \\
S' \ar[r]^-{\iota'_S} & \mathbb{A}_g.
}
\]
Let $S''$ be a regular, Zariski open and dense subset of
$\iota'_S(S')^{\mathrm{Zar}}$ that is contained in the image of $S'$.
Let $\cA'' := \mathfrak{A}_g \times_{\mathbb{A}_g} S''$ and $X'' := \iota'(X') \cap \cA''$. Then $\dim X' \ge \dim X''$. It is not hard to check that 
$\mathbb{Z}X''$ is Zariski dense in $\cA''$, and $X''(\IQbar) \cap
\cA''_{\mathrm{tor}}$ is Zariski dense in $X''$. \textit{If}
Theorem~\ref{MainThm} holds under the extra assumption {\tt (Hyp)}
then  it holds true for $X'' \subseteq \cA'' \rightarrow S''$. So $\dim X'' \ge g$. Hence $\dim X = \dim X' \ge \dim X'' \ge g$, and we are done.

Therefore, from now on without loss of generality we work with $\cA
\rightarrow S$ satisfying {\tt (Hyp)}. 

Let $X \subseteq \cA$ be an irreducible subvariety defined over
$L=\IQbar$ satisfying the assumptions of Theorem~\ref{MainThm} and {\tt (Hyp)}. Our goal is to prove $\dim X\ge g$. We proceed by
induction on $g\ge 1$. 

First observe that  $\dim X \ge 1$. Indeed, if $X$ were a point, it
would be torsion. But this contradicts the hypothesis of
Theorem~\ref{MainThm}.
We thus recover the case $g=1$. 

We assume $g\ge 2$ and
that Theorem~\ref{MainThm} holds true for all $\cA \rightarrow
S$ satisfying {\tt (Hyp)} of relative dimension
$1,\ldots,g-1$. 


By Theorem~\ref{TheoremFirstDegLocusIsX}, $X_\IC^{\mathrm{deg}}(1)$ is
Zariski dense in $X_\IC$. So $X_\IC^{\mathrm{deg}}(1) = X_\IC$ by the Zariski
closedness of $X_\IC^{\mathrm{deg}}(1)$,
Theorem~\ref{ThmDegLocusZarClosed}. Now we apply
\cite[Theorem~8.1]{GaoBettiRank} with $t = 1$. 
 Indeed, the hypothesis on 
$\cA_X=\pi^{-1}(S)$ of the reference is satisfied since $\IZ X$ is
Zariski dense in $\cA$. 
We get a quotient abelian
scheme $\varphi \colon \cA \rightarrow \cB'$ of relative dimension
$g'$, \textit{i.e.}, there exists an abelian subscheme $\cB$ of $\cA
\rightarrow S$ with $\varphi$ being the quotient $\cA \rightarrow
\cA/\cB$, such that for the diagram
    \begin{equation*}
      \xymatrix{
        \cA \ar[d]\ar[r]^{\varphi}&      \cB' \ar[d] \ar[r]^-{\iota} \pullbackcorner & \mathfrak{A}_{g'}\ar[d] \\
        S \ar@{=}[r]&      S \ar[r]^-{\iota_S} & \mathbb{A}_{g'},
      }
    \end{equation*}
we have $\dim (\iota\circ\varphi)(X) < \dim X - (g-g') + 1$ and that $\iota\circ \varphi$ is not generically finite.

We have $0 \le \dim (\iota\circ\varphi)(X) \le \dim X - (g-g')$. So if
$g'=0$ then $\dim X \ge g$ and we are done. Let us now assume $g'\ge
1$. 

We claim that $g' < g$. Indeed, otherwise $g' = g$ and $\varphi$ is the identity. But then $\iota$ is the inclusion by {\tt (Hyp)}. This contradicts the fact that $\iota\circ \varphi$ is not generically finite.

Now $\iota(\cB') \rightarrow \iota_S(S)$ is an abelian scheme of
relative dimension $g' \le g-1$, and $(\iota\circ\varphi)(X)$ is an
irreducible subvariety of $\iota(\cB')$.
Recall that $\IZ X$ is Zariski dense in $\cA$. It is not hard to check that $(\iota\circ\varphi)(X)$ dominates $\iota_S(S)$, $\mathbb{Z}(\iota\circ\varphi)(X)$ is Zariski dense in $\iota(\cB')$, and $(\iota\circ\varphi)(X) \cap \iota(\cB')_{\mathrm{tor}}$ is Zariski dense in $(\iota\circ\varphi)(X)$. In other words,  $(\iota\circ\varphi)(X) \subseteq \iota(\cB') \rightarrow \iota_S(S)$ still satisfy the assumptions of Theorem~\ref{MainThm}.

We are almost ready to apply the induction hypothesis, except that $\iota_S(S)$ may not be regular. Set $S_0 := \iota_S(S)^{\mathrm{reg}}$, $\cB'_0 := \iota(\cB') \times_{\iota_S(S)} S_0$ and $X_0 := (\iota\circ\varphi)(X) \cap \cB'_0$. Then $X_0$ is Zariski open dense in $(\iota\circ\varphi)(X)$, and $X_0 \subseteq \cB'_0 \rightarrow S_0$ still satisfy the assumptions of Theorem~\ref{MainThm}.

The relative dimension of $\cB'_0 \rightarrow S_0$ is $g' \le g-1$. So we can apply the induction hypothesis and get $\dim X_0 \ge g'$. So $\dim X > \dim (\iota\circ\varphi)(X) + g-g'-1 = \dim X_0 + g-g'-1 \ge g-1$. Therefore $\dim X \ge g$ and we are done.
\end{proof}


%% file: UnifMM.tex
In this section, we explain how Theorem~\ref{MainThm} for $L=\IQbar$
implies Corollary~\ref{CorUniformMM}.

Let $\mathbb{M}_g$ be the moduli space of smooth projective
geometrically irreducible curves of genus $g\ge 2$ with symplectic
level-$3$-structure. Let $\mathfrak{C}_g \rightarrow \mathbb{M}_g$
be the universal curve. Let $\mathfrak{J}_g =
\mathrm{Jac}(\mathfrak{C}_g/\mathbb{M}_g) = \mathrm{Pic}^0(\mathfrak{C}_g/\mathbb{M}_g)$ be the relative Jacobian.
As usual, everything is in characteristic $0$ and we take
$\mathfrak{C}_g$ to be defined over a number field and geometrically
irreducible.

For the proof of Corollary~\ref{CorUniformMM} we treat familes of
curves over $\mathbb{M}_g$. The proof of this
lemma is inspired by \cite[Theorem~2.4]{Stoll:Uniform} and \cite[Proposition~7.1]{DGHUnifML}. For each subvariety $S$ of $\mathbb{M}_g$, write
$\mathfrak{C}_S$ and $\mathfrak{J}_S$ for the base changes
$\mathfrak{C}_g\times_{\mathbb{M}_g}S$ and
$\mathfrak{J}_g\times_{\mathbb{M}_g}S$.

\begin{lemma}\label{LemUML}
  Let $S$ be a regular, irreducible, locally closed subvariety of
  $\mathbb{M}_g$. Then there exists a constant $c = c(S) > 0$ such that
  for all $s\in S(\IQbar)$ there is $\Xi_s\subset
  \mathfrak{C}_s(\IQbar)$ with $\#\Xi_s < c$ and with the following
  property.
  For all $x\in \mathfrak{C}_s(\IQbar)\ssm \Xi_s$ we have
    $\#\{y \in \mathfrak{C}_s(\IQbar) : y-x \in (\mathfrak{J}_s)_{\mathrm{tor}} \} < c$.
\end{lemma}

\begin{proof}
We prove the lemma by induction on $\dim S$. The proof for the base step $\dim S =0$ is in fact contained in the induction.

Alternatively, the base case $\dim S =0$ follows from work of Baker
and Poonen as follows. If $S$ is a point $s$, we are dealing with a single
curve. We may take $\Xi_s=\emptyset$. The desired bound follows from
\cite[Corollary~3]{BakerPoonen:01} on torsion packets.

Set $ \mathfrak{C}_g^{[6]} $ to be the $6$-th fibered power of
$\mathfrak{C}_g$ over $\mathbb{M}_g$, and $\mathfrak{J}_g^{[5]}$ to be
the $5$-th fibered power of $\mathfrak{J}_g$ over $\mathbb{M}_g$. Let
$\mathscr{D}_5 \colon \mathfrak{C}_g^{[6]} \rightarrow
\mathfrak{J}_g^{[5]}$ be the Faltings--Zhang map, fiberwise defined by
sending $(x_0,x_1,\ldots,x_5) \mapsto (x_1-x_0,\ldots,x_5-x_0)$. Write
$\mathfrak{C}_S^{[6]}$ and $\mathfrak{J}_S^{[5]}$ for the base changes
$\mathfrak{C}_g^{[6]}\times_{\mathbb{M}_g}S$ and
$\mathfrak{J}_g^{[5]}\times_{\mathbb{M}_g}S$.

Consider $X= \mathscr{D}_5(\mathfrak{C}_S^{[6]})$, which is an
irreducible subvariety of $\mathfrak{J}_S^{[5]}$.

Observe that $\dim X\le 6+\dim S$. Since $g \ge 2$, we have $$ \dim
{\mathfrak{J}_S^{[5]}} - \dim X \ge 5g-6 > 3g-3 = \dim
\mathbb{M}_g\ge\dim S.$$
We conclude $\dim X < \dim \mathfrak{J}_S^{[5]}-\dim S$.

We claim that $\ZZ X = \bigcup_{N \in \ZZ} [N]X$ is Zariski dense in
$\mathfrak{J}_S^{[5]}$. Indeed, it suffices to prove that for each $s
\in S(\IQbar)$, the fiber $X_s$ is not contained in any proper
algebraic subgroup of $\mathfrak{J}_s^5$. Now take  $x_0 \in \mathfrak{C}_s(\IQbar)$, then $X_s = \mathscr{D}_5(\mathfrak{C}_s^6) \supseteq \mathscr{D}_5(\{x_0\} \times \mathfrak{C}_s^5) = (\mathfrak{C}_s-x_0)^5$, with $\mathfrak{C}_s - x_0$ viewed as the image of the Abel--Jacobi map from $\mathfrak{C}_s$ to $\mathfrak{J}_s$ based at $x_0$. Since $\mathfrak{C}_s - x_0$ generates $\mathfrak{J}_s$, we conclude that $X_s$ is not contained in any proper subgroup of $\mathfrak{J}_s^5$.


Thus we can apply Theorem~\ref{MainThm} to $X$  in
$\mathfrak{J}_S^{[5]} \rightarrow S$ and $L=\IQbar$. We conclude that 
\[
Y = \overline{X \cap (\mathfrak{J}_S^{[5]})_{\mathrm{tor}}}^{\mathrm{Zar}}
\]
is a proper, Zariski closed subset of $X$.

Let $\pi \colon \mathfrak{J}_S^{[5]} \rightarrow S$ denote the
structure morphism.
Now $\dim \overline{S \setminus
  \pi(X\setminus Y)}^{\mathrm{Zar}}< \dim S$.
Thus $\overline{S \setminus
  \pi(X\setminus Y)}^{\mathrm{Zar}}$ is a finite union of regular,
irreducible, locally closed subvarieties $S'$ of $S$ with $\dim
S'<\dim S$. 
So we can apply induction hypothesis to each of these $S'$ and
conclude the result. Thus, in order to prove the
lemma, it suffices to consider $s\in S(\IQbar)$  in
$\pi(X\setminus Y)$. Thus we have 
\begin{equation*}
Y_s \subsetneq X_s = \mathscr{D}_5(\mathfrak{C}_s^{[6]}). 
\end{equation*}
Let $x \in \mathfrak{C}_s(\IQbar)$.
We are in one of the two cases:
\begin{enumerate}
\item[(i)] $(\mathfrak{C}_s-x)^5 \subseteq Y_s$ or 
\item[(ii)] $(\mathfrak{C}_s-x)^5 \not\subseteq Y_s$.
\end{enumerate}

In case (i), we apply \cite[Lemma 6.4]{DGHUnifML} with
$Z$ any irreducible component of $Y_s\subsetneq
\mathscr{D}_5(\mathfrak{C}_s^6)$. We conclude that
there are at most $84(g-1)\cdot n(Y_s)$ possibilities
for $x$ where $n(Y_s)$ is the number of irreducible
components of $Y_s$. Notice that
$n(Y_s)$ is strictly less than a number $c_1'$ independent of $s$. Set $c_1 = 84(g-1)c_1'$.
The $x$ leading to case (i) will amount to the points in $\Xi_s$. 

In case (ii), set $\Sigma = \{y \in \mathfrak{C}_s(\IQbar) : y-x \in
(\mathfrak{J}_s)_{\mathrm{tor}} \}$. By \cite[Lemma 6.3]{DGHUnifML},
there exists a number $c_2$, independent of $s$, such that the
following holds. If $\# \Sigma  \ge c_2$, then $(\Sigma-x)^5 \not\subseteq Y_s(\IQbar)$. But $(\Sigma-x)^5 = \mathscr{D}_5(\{x\} \times \Sigma^5) \subseteq X(\IQbar)$ and $(\Sigma-x)^5 \subseteq (\mathfrak{J}_s^{[5]})_{\mathrm{tor}}$, so $(\Sigma-x)^5 \subseteq Y_s(\IQbar)$ by definition of $Y$. Therefore $\#\Sigma < c_2$.

Now the lemma follows by taking $c = \max\{c_1,c_2\}$.
\end{proof}

Lemma~\ref{LemUML} improves itself as follows.

\begin{proposition}\label{PropUML}
  Let $S$ and $c=c(S)$ be as in Lemma~\ref{LemUML}. For
  all $s\in S(\IQbar)$ and all $x \in \mathfrak{C}_s(\IQbar)$ we
  have  $\#\{y \in \mathfrak{C}_s(\IQbar) : y-x \in (\mathfrak{J}_s)_{\mathrm{tor}} \} < c$.
\end{proposition}
\begin{proof}
  Let $s\in S(\IQbar)$ and $x\in \mathfrak{C}_s(\IQbar)$.
  Suppose $y_1,\ldots,y_n\in\mathfrak{C}_s(\IQbar)$ are pairwise
  distinct with $y_i-x$ torsion for all $i$.
  If $n\ge c$, then some $y_i$, say $y_1$, lies in
  $\mathfrak{C}_s(\IQbar)\ssm \Xi_s$ with $\Xi_s$ the set from
  Lemma~\ref{LemUML}.
  Now $y_i-y_1 = (y_i-x)-(y_1-x)$ are $n$  torsion points
  of $\mathfrak{J}_s(\IQbar)$ for $i\in\{1,\ldots,n\}$. So
  Lemma~\ref{LemUML} implies $n<c$, a contradiction. We conclude $n<c$. 
\end{proof}

\begin{proof}[Proof of Corollary~\ref{CorUniformMM}]
  By a specialization argument \cite[Lemma~3.1]{DGHBog} based on
  Masser's \cite{masser1989specializations}, it suffices to prove the
  result with $\IC$ replaced by $\IQbar$. Then the result follows
  immediately from Propsition~\ref{PropUML} applied to $S =
  \mathbb{M}_g$.
\end{proof}


%% file: BettiSubmersive.tex
We prove Theorem~\ref{ThmCritTorsionDense} for $L=\IQbar$ in this
section. This proof uses Theorem~\ref{MainThm} for $L=\IQbar$, which
has already been establised in $\mathsection$\ref{SectionConclusion},
and the criterion of $\mathrm{rank}_{\mathrm{Betti}}(X) < 2g$ by
\cite[Theorem~1.1]{GaoBettiRank} (with $l = g$).
Recall that $\cA$ is an abelian scheme of relative dimension $g\ge 1$
over a regular, irreducible, quasi-projective base $S$ defined over $\IQbar$. 


\begin{proof}[Proof of Theorem~\ref{ThmCritTorsionDense} for $L=\IQbar$]
The implication ``$\Leftarrow$'' follows from
\cite[Proposition~2.1.1]{ACZBetti} for subfields of $\IC$.
Indeed,
$X(\IC)\cap \cA_{\mathrm{tors}}$ is dense in $X^{\mathrm{an}}$.
Among these point we may find a Zariski dense subset of
$\IQbar$-points as follows. 
The
Betti map is locally injective. So we may find a Zariski dense set of
isolated intersection points in $X(\IC)\cap \ker[n]$ as $n\in\IN$
varies.
These points are defined over $\IQbar$ as $X$ is. 

From now on, we focus on ``$\Rightarrow$''. Let $\cA \rightarrow S$ be an abelian scheme defined over $L=\IQbar$, and let $X$ be an irreducible subvariety of $\cA$ defined over $L$ such that $\ZZ X$ is Zariski dense in $\cA$ and that $X(L) \cap \cA_{\mathrm{tor}}$ is Zariski dense in $X$. 
(Observe that Theorem~\ref{MainThm} in the already proved case
$L=\IQbar$ now implies $\dim X\ge g$.)

Assume $\mathrm{rank}_{\mathrm{Betti}}(X) < 2g$. We wish to get a contradiction.

By \cite[Theorem~1.1]{GaoBettiRank} applied to $l = g$, there exists a quotient abelian scheme $\varphi \colon \cA \rightarrow \cB'$ of relative dimension $g'$, \textit{i.e.}, there exists an abelian subscheme $\cB$ of $\cA \rightarrow S$ with $\varphi$ being the quotient $\cA \rightarrow \cA/\cB$, such that for the diagram
    \begin{equation*}
      \xymatrix{
        \cA \ar[d]\ar[r]^{\varphi}&      \cB' \ar[d] \ar[r]^-{\iota} \pullbackcorner & \mathfrak{A}_{g'}\ar[d] \\
        S \ar@{=}[r]&      S \ar[r]^-{\iota_S} & \mathbb{A}_{g'},
      }
    \end{equation*}
we have $\dim (\iota\circ\varphi)(X) < g'$. 

Set $S_0 := \iota_S(S)^{\mathrm{reg}}$, $\cA_0' := \mathfrak{A}_{g'} \times_{\mathbb{A}_{g'}} S_0$ and $X_0 := (\iota\circ \varphi)(X) \cap \cA_0'$. Then $\cA_0' \rightarrow S_0$ is an abelian scheme of relative dimension $g'$ defined over $L$, and $X_0$ is Zariski open dense in $(\iota\circ\varphi)(X)$. So $\dim X_0 < g'$.

It is not hard to check that $\ZZ X_0$ is Zariski dense in $\cA_0'$
since $\ZZ X$ is Zariski dense in $\cA$, and that $X_0(L) \cap
\cA'_{0,\mathrm{tor}}$ is Zariski dense in $X_0$ since $X(L) \cap
\cA_{\mathrm{tor}}$ is Zariski dense in $X$. Thus we can apply
Theorem~\ref{MainThm}, over the base field $L$,
to $X_0 \subseteq \cA'_0 \rightarrow S_0$ and conclude that $\dim X_0 \ge g'$.

The conclusions of the last two paragraphs are contradictory. Thus we get the desired contradiction. 
So $\mathrm{rank}_{\mathrm{Betti}}(X) = 2g$. We are done.
\end{proof}

Let $L$ be an algebraically closed subfield of $\IC$.
The argument above shows that   Theorem \ref{MainThm} for $L$ implies Theorem
\ref{ThmCritTorsionDense} for $L$.


%% file: Specialization.tex
The goal of this section is to prove Theorem~\ref{MainThm}, proved in
$\mathsection$\ref{SectionConclusion} for $L=\IQbar$, for an arbitrary
algebraically closed field $L$ of characteristic $0$. Note first, that
it suffices to consider only subfields $L$ of $\IC$ by a suitable
Lefschetz principle. Moreover, as all varieties are defined using
finitely many polynomials and coefficients, we may assume that
$\mathrm{trdeg}_\IQbar L <\infty$. 

In this section we will also complete the proof of Theorem~\ref{ThmCritTorsionDense}.

Let $L$ be an algebraically closed subfield of $\IC$ that has finite
transcendence degree over $\IQbar$. Let $\cA
\rightarrow S$ be an abelian scheme defined over $L$, and let $X
\subseteq \cA$ be an irreducible subvariety. We suppose that $\ZZ X$
is Zariski dense in $\cA$ and $X(\IC) \cap \cA_{\mathrm{tor}}$ is
Zariski dense in $X$.

We shall prove both Theorem~\ref{MainThm} and
Theorem~\ref{ThmCritTorsionDense} by induction on $
\mathrm{trdeg}_{\IQbar} L$. We shall proceed as follows. For each
integer $d \ge 0$, we set
\begin{center}
  {\tt RMM(d):} Theorem~\ref{MainThm} holds true if
  $\mathrm{trdeg}_\IQbar L\le d$ 
\end{center}
and
\begin{center}
  {\tt TorDense(d):} Theorem~\ref{ThmCritTorsionDense} holds true if
  $\mathrm{trdeg}_\IQbar L\le d$. 
\end{center}
We proceed by proving the following statements
for each integer $d \ge 0$:
\begin{enumerate}
\item[(i)] {\tt RMM(d)} implies {\tt TorDense(d)};
\item[(ii)] {\tt TorDense(d)} implies {\tt RMM(d+1)}.
\end{enumerate}
Note that the
statement ${\tt RMM(0)}$ was proved in
$\mathsection$\ref{SectionConclusion}, \textit{i.e.},
Theorem~\ref{MainThm} in the case $L=\IQbar$.

\subsection{Proof of {\tt RMM(d)}$\Rightarrow${\tt TorDense(d)}}
This follows from a verbalized copy of the proof executed in $\mathsection$\ref{SectionCriterionTorsionDenseOverIQbar}.

\subsection{Proof of {\tt TorDense(d)}$\Rightarrow${\tt RMM(d+1)}}
Suppose we are in the case $\mathrm{trdeg}_{\IQbar} L = d+1$. So
$\cA,S,$ and $X$ are all defined over $L$. We assume that $\IZ X$ is
Zariski dense in $\cA$ and
$X(L)\cap\cA_{\mathrm{tors}}$ is Zariski dense in $X$. In particular,
$S=\pi(X)$.
Our goal is to
show $\dim X \ge g$ with $g$ the relative dimension of $\cA/S$.

To start we follow the argument from the beginning of the proof of
Theorem~\ref{MainThm} to reduce to the universal family.
We obtain from $X$ a regular, irreducible, locally closed $S''\subset\IA_{g,L}$
as well as an irreducible subvariety  $X''\subset
\cA''=\mathfrak{A}_{g,L}\times_{\IA_{g,L}} S''$
such that $\IZ X''$ is Zariski dense in $\cA''$
and such that $X''(L)\cap\cA''_{\mathrm{tors}}$ is Zariski dense in
$X''$. Moreover, $\dim X'' \le \dim X$. If we can establish $\dim
X''\ge g$ then we are done.

From now on we assume that $X$ is an irreducible  subvariety of
$\mathfrak{A}_{g,L}$ and $\pi(X) = S\subset\IA_{g,L}$. Recall that we
already reduced to the case $L\subset\IC$.

There exists an algebraically closed subfield $K \subseteq L$ with 
 $K\supset \IQbar$ such that $\mathrm{trdeg}_{\IQbar}K = d$. Then
$\mathrm{trdeg}_K L = 1$.

By \cite[Lemma~2.2]{BarDill} there is an irreducible  subvariety
$\mathfrak{X}\subset \mathfrak{A}_{g,K}$ such that $X\subset
\mathfrak{X}_L$ and $\dim \mathfrak X \le \dim X + 1$. Note that $\dim
X\le \dim\mathfrak X$. We may assume that $\mathfrak X$ is the
minimal subvariety of $\mathfrak{A}_{g,K}$ whose base change to $L$
contains $X$.

If $\dim \mathfrak X =\dim X$, then $X$ was originally already
defined over $K$. In this case $\dim X\ge g$ follows from \texttt{RMM(d)}.
So we may assume $\dim \mathfrak X = \dim X + 1$. 

Let $\mathfrak S = \pi(\mathfrak X)\subset\IA_{g,K}$.
Then $\mathfrak{S}_L\supset S$.
We claim that $\IZ \mathfrak X$ is Zariski dense in
$\pi^{-1}(\mathfrak S)$.
The Zariski closure $Z$ of $\IZ \mathfrak X$ in
$\pi^{-1}(\mathfrak S)$ 
has dimension at least
$g+\dim S = \dim \pi^{-1}(S)$ since $Z_L\supset \overline{\IZ
  X}^{\mathrm{Zar}}$.
Our claim follows if $\dim S =
\dim\mathfrak{S}$.  So we may assume $1+\dim S \le  \dim \mathfrak S$.
Suppose $g+\dim S = \dim Z$, then  $Z_L= \pi^{-1}(S)$
which contradicts $ \pi(Z)=\mathfrak{S}$.
 Thus $\dim Z\ge g + \dim S + 1$.
By \cite[Lemma~2.2]{BarDill} $S$ is contained in $\mathfrak{S}'_L$
with $\mathfrak{S}'\subset\IA_{g,K}$ and $\dim \mathfrak{S}' \le \dim
S + 1$. So $\mathfrak{X} \subset\pi^{-1}(\mathfrak S')$ by minimality
of $\mathfrak{X}$. Thus $\mathfrak S \subset \mathfrak S'$ and in
particular $\dim \mathfrak{S} = \dim S + 1$.
So we must have $\dim Z \ge
g+\dim S + 1 = g+\dim\mathfrak S = \dim \pi^{-1}(\mathfrak S)$. This
implies $Z=\pi^{-1}(\mathfrak S)$.

We claim that $\mathfrak{X}_\IC^{\mathrm{deg}}(0)$ lies Zariski dense
in $\mathfrak{X}$.
Let $x \in X(L) \cap (\mathfrak{A}_{g,L})_{\mathrm{tors}}$.
The $K$-Zariski closure $C$ of $\{x\}$ in $\mathfrak{X}$ is contained in the kernel of
$[\mathrm{ord}(x)]$.
We have $\dim C\le 1$ since $\mathrm{trdeg}_K L = 1$, indeed, use
again \cite[Lemma~2.2]{BarDill}. As 
$X(L) \cap (\mathfrak{A}_{g,L})_{\mathrm{tors}}$ is Zariski dense in
$X$ and as $X$ is not defined over $K$, we may assume $\dim C = 1$. 
Such a curve lies in $\mathfrak{X}_\IC^{\mathrm{deg}}(0)$ as it is torsion. Ranging
over the possible $x$ yields our claim.

By \cite[Theorem~1.7]{GaoBettiRank} we conclude
$\mathrm{rank}_{\mathrm{Betti}}(\mathfrak{X}) < 2\dim \mathfrak{X}$.

The above argument also implies that
$\mathfrak{X}(K)\cap(\mathfrak{A}_{g,K})_{\mathrm{tors}}$ is Zariski
dense in $\mathfrak X$. We apply \texttt{TorDense(d)} to
$\mathfrak{X}\cap\pi^{-1}(\mathfrak{S}^{\mathrm{reg}})$, which is
defined over $K$, and conclude
$\mathrm{rank}_{\mathrm{Betti}}(\mathfrak{X}) =2g $.

Combining both bounds for the Betti rank yields $\dim\mathfrak X > g$.
As $\dim X + 1 = \dim \mathfrak{X}$ we conclude $\dim X\ge g$, as desired. \qed

%% file: Appendix.tex
In $\mathsection$\ref{SectionLGO} we used a quantitative version,
obtained by R\'{e}mond \cite[Proposition~2.9]{RemondGaloisBound}, of
Masser's earlier result on the Galois orbit of torsion points on
abelian varieties. In this appendix, we prove in
Proposition~\ref{PropDavidBound} an estimate that is sufficient for
our purposes. We will rely on a result of David
\cite[Th\'eor\`eme~1.4]{DavidMinHaut}, it is not directly related to
Masser and W\"{u}stholz's Isogeny Theorem \cite{MW:abelianisog}.

Throughout this section let $k$ be a number field contained in a fixed algebraic closure
$\overline k$. 
If not stated otherwise,  $A$ denotes
an abelian variety of dimension $g\ge 1$  defined over  $k$. We set
\[
  \rho(A,k) = 
  [k:\IQ](\max\{1,h_{\mathrm{Fal}}(A)\}+\log [k:\IQ])\ge [k:\IQ]
\]
where
$h_{\mathrm{Fal}}(A)$ denotes the \textit{stable Faltings height} of
$A$. The additive normalization in $h_{\mathrm{Fal}}$ plays no role in
the current work.

We begin by stating a special case of \cite[Th\'eor\`eme 1.4]{DavidMinHaut}. 

\begin{theorem}[David]
  \label{thm:david}
  For each integer $g\ge 1$  there is a constant $c_1(g)>0$
  with the following property. 
  Let $(A,\cL)$ be
  a principally polarized abelian variety of dimension $g$
  defined over $k$. Suppose $x\in A(k)$
  has finite order. Then there exists an abelian subvariety
  $B\subsetneq A$ defined over $\overline k$ and  positive integer $N$ such that
  $[N](x)\in B(\overline k)$ and  $\max\{N,\deg_\cL B\} \le c_1(g) \rho(A,k)^{g+1}$.
\end{theorem}
\begin{proof}
  Since $x$ has finite order, its N\'eron--Tate height vanishes. So we
  are in the second case of \cite[Theorem~1.4]{DavidMinHaut}.
  David used a height of $A$ defined using theta functions
  (and thus is sometimes called a \textit{Theta height}). The comparison
  with  $h_{\mathrm{Fal}}(A)$ is done
  by work of Bost and David; see \cite[Corollaire~6.9]{DPvarabII} and
  \cite[Corollary~1.3]{Pazuki:12}. Thus in \cite[Th\'eor\`eme~1.4]{DavidMinHaut}, we can use the stable Faltings height after
  modifying the multiplicative constant $c_1(g)$ (which is called $c_2$
  in the reference.)
  Moreover, in the reference we  may assume $\|\mathrm{Im}\tau\|\ge
  \sqrt{3}/2$ as $\tau$ can be assumed to be Siegel reduced.
  Set $D = \max\{2,[k:\IQ]\}$ and $h = \max\{1,h_F(A)\}$. 
  Observe that
  $$
  \frac{2}{\sqrt 3} D(h+\log D) + D^{1/(g+2)} \le 2 D(h+\log D) + D\le
  3 D(h+\log D)
  $$
  as $h+\log D\ge 1$. If $k\not=\IQ$, then $D(h+\log D) = \rho(A,k)$
  and otherwise $D(h+\log D) = 2(h+\log 2) \le 2(1+\log 2) h \le 4 \rho(A,\IQ)$. 
  So the desired bound follows after
  modifying $c_1(g)$ again.
\end{proof}

We will reduce to the principally polarized case using the next lemma.


\begin{lemma}
  \label{lem:ppreduction}
  Let $(A,\cL)$ be a polarized abelian
  variety with $\dim A =g\ge 1$ and
  $\Delta = (\deg_\cL A)/g!$. Then $\Delta = \dim H^0(A,\cL)$ and 
  there is a number field $k'\supset k$ 
  with $[k':k]\le \Delta^{2g}$,
  a principally polarized abelian variety $(B,\cM)$ defined over
  $k'$, and an isogeny $A_{k'}\rightarrow B$, also defined over
  $k'$, of degree at most $\Delta$. Moreover,
  $\rho(B,k')\le \Delta^{2g}(\rho(A,k) + 3g[k:\IQ]\log\Delta)$. 
\end{lemma}
\begin{proof}
  The Riemann--Roch Theorem and
  Vanishing Theorem on higher cohomology for ample line bundles,
  see \cite[Section 16]{MumfordAbVar70}, imply $\Delta=\dim H^0(A,\cL) = 
  (\deg_\cL A)/g!$. This is the first claim.
  By a classical result, see \cite[Lemme~3.5]{GR:periodsisgoenies},
  there is an isogeny $\varphi\colon A_{\overline k}\rightarrow B$ of degree
  $\dim H^0(A,\cL) =\Delta$
  to a principally polarized abelian variety $(B,\cM)$
  defined over $\overline k$; here $\varphi^* \cM = \cL$. 

  The kernel $\ker \varphi$ is isomorphic to a product of cyclic 
  groups of order $t_1,\ldots,t_s$, say. Both $B$ and $\varphi$ are
  defined over the field $k'$ field generated by certain points of order
  $t_1,\ldots,t_s$, respectively. A point in $A(\overline k)$ of order
  $t_i$ generates a field extension of $k$ of degree at most $t_i^{2g}$.
  Therefore, $[k':k]\le (t_1\cdots t_s)^{2g} = \Delta^{2g}$.

  Faltings's estimate implies
  \begin{equation*}
    h_F(B)\le h_F(A) + \frac 12 \log\Delta. 
  \end{equation*}
  So
  \begin{alignat*}1
    \rho(B,k') &= [k':\IQ] (\max\{1,h_F(B)\}+\log[k':\IQ]) \\
    &\le \Delta^{2g}[k:\IQ] \left(\max\{1,h_F(A)\}+\frac 12\log\Delta +
    2g\log\Delta +\log[k:\IQ]\right)\\
    &=\Delta^{2g} \rho(A,k)+\Delta^{2g}[k:\IQ]\left(\frac 12+2g\right)\log\Delta. \qquad \qquad \qedhere
  \end{alignat*}
\end{proof}

The following lemma is obtained by iterating David's result. For
an integer $g\ge 1$ we define $\lambda(g) = 3^{g+1}(g!)^2 \ge 3$. 

\begin{lemma}\label{lem:DavidBound}
  For each integer $g\ge 1$  there is a constant $c_2(g)>0$
  with the following property.   
  Let
  $(A,\cL)$ be
  a {principally polarized} abelian variety of dimension $g$
  defined over $k$. 
  If $x\in A(k)$ has finite order, then
   $\ord{x}\le c_2(g) \rho(A,k)^{\lambda(g)}$.
\end{lemma}
\begin{proof}
  The field generated by $k$ and all $3$-torsion points of $A$ has
  degree at most $3^{(2g)^2}$ over $k$. After replacing $k$  by this
  field  we may assume that all $3$-torsion points are
  $k$-rational. By Silverberg \cite[Theorem~2.4]{Silverberg:fielddef}
  and Poincar\'e's Complete Reducibility Theorem,  all abelian subvarieties of $A$ are
  defined over $k$. 

  Below $c_2(g),c_3(g),\ldots$ denote positive values that depend only
  on $g$. We will determine $c_2(g)$ by induction on $g\ge 1$.
  The case $g=1$ is treated during the induction.

  Suppose $x\in A(k)$ has finite order. 
  We let $B$ and $N$ be as in  Theorem~\ref{thm:david} applied to $x$.
  If $B=0$, then $x$ has order at most $N$
  and we are done.
  In the case $g=1$ then $B=0$ since $B\subsetneq A$. So this argument
  covers the base case of the induction.

  We assume $\dim B\ge 1$.
  Note that $B$ need not be principally polarized. Let $k'$ and $C$ be
  as in Lemma~\ref{lem:ppreduction} applied to $(B,\cL|_B)$. We define
  \begin{equation}
    \label{eq:Deltabound}
     \Delta = \dim H^0(B,\cL|_B) = \frac{\deg_\cL B}{(\dim B)!}\le c_1(g) \rho(A,k)^{g+1} 
  \end{equation}

  It is known that
  the stable Faltings height satisfies
  \begin{equation*}
    h_F(B) \le h_F(A) + \log \dim H^0(B,\cL|_B) + c_3(g) = h_F(A) +
    \log\Delta + c_3(g),
  \end{equation*}
  see \cite[$\mathsection$2.3]{GR:periodsisgoenies}. This and (\ref{eq:Deltabound}) imply
  \begin{equation*}
    \rho(B,k) \le [k:\IQ] ( \max\{1,h_F(A)\} + \log \Delta +c_3(g) +
    \log[k:\IQ])
    \le c_4(g) \rho(A,k). 
  \end{equation*}



  We write $\varphi\colon B_{k'}\rightarrow C$ for the isogeny provided by
  Lemma~\ref{lem:ppreduction}. 
  As $[N](x)\in B(k')$, the image $\varphi([N](x))\in C(k')$
  is well-defined and of finite order. 
  Moreover, $\dim C=\dim B < g$, so this lemma applied by
  induction to the pair $C,k'$ and  $\varphi([N](x))$
  together with the
  estimates above
  yields
  \begin{alignat*}1
    \ord{\varphi([N](x))} &\le  c_2(\dim B)
    \rho(C,k')^{\lambda(\dim B)} \\
    &\le c_5(g) c_2(\dim B) \left(\Delta^{2\dim B}(\rho(B,k)+[k:\IQ]
      \log\Delta)\right)^{\lambda(\dim B)}
    \\
    &\le c_6(g) c_2(\dim B)
    \left(\Delta^{2\dim B}(c_4(g)\rho(A,k)+[k:\IQ] \log\rho(A,k))\right)^{\lambda(\dim B)}
    \\
    &\le c_7(g) c_2(\dim B) \left(\Delta^{2\dim B}\rho(A,k)\right)^{\lambda(\dim B)}\\
    &\le c_8(g) c_2(\dim B) \rho(A,k)^{(1+2(g+1)\dim B)\lambda(\dim B)}. 
  \end{alignat*}
  
  We use the bound for $N$ from Theorem~\ref{thm:david} and the bound for
  $\deg\varphi$ from Lemma~\ref{lem:ppreduction} together with
  (\ref{eq:Deltabound}) to find
  \begin{alignat*}1
    \ord{x} &\le 
    N(\deg\varphi)
    \ord{\varphi([N]x)} \\
    &\le c_1(g)^2 \rho(A,k)^{2(g+1)}
    \ord{\varphi([N]x)}
    \\
    &\le c_{9}(g) c_2(\dim B) \rho(A,k)^{2(g+1)+(1+2(g+1)\dim B)\lambda(\dim
      B)}.
  \end{alignat*}

  As $\dim B\le g-1$ and $\lambda(\dim B)\le \lambda(g-1)$
the exponent of $\rho(A,k)$ is at most
  \begin{equation*}
    2(g+1)+(2g^2-1)\lambda(g-1)\le 2g +2g^2\lambda(g-1);
  \end{equation*}
  the final step used $\lambda(g-1)\ge 2$. The proposition follows as
  $2g+2g^2\lambda(g-1) = 2g+2 \cdot 3^g g^2((g-1)!)^2=2g+2\cdot 3^g
  (g!)^2 \le \lambda(g)$.
  %
%
%
\end{proof}

Finally, we drop the principally polarized hypothesis.
\begin{proposition}\label{PropDavidBound}
  For each integer $g\ge 1$ there is a constant $c(g)>0$ with the
  following property. Let $(A,\cL)$ be a
  polarized abelian variety of dimension $g$ defined over $k$. If $x\in
  A(k)$ has finite order, then
  $$\ord{x}\le c(g)
  (\deg_\cL A)^{1+(2g+1)\lambda(g)} \rho(A,k)^{\lambda(g)}.$$
\end{proposition}
\begin{proof}
  Let $k',B,$ and $\varphi\colon A_{k'}\rightarrow B$ be as in
  Lemma~\ref{lem:ppreduction} applied to $(A,\cL)$.
  Say $x\in A(k')$ has finite order. Then so does $\varphi(x) \in B(k')$ and
  $\ord{x}\le \#(\ker\varphi)\cdot \ord{
    \varphi(x)} \le \frac{1}{g!}(\deg_\cL A) \cdot \ord{\varphi(x)}$.
  We apply Lemma~\ref{lem:DavidBound} to bound 
  $\ord{\varphi(x)}$ from above and conclude the proposition after a
  short calculation.
\end{proof}

The value of the exponent $\lambda(g)$ is irrelevant for the
method. It suffices to know that it is a function of $g$.



%% file: main.bbl
\def\cprime{$'$}
\begin{thebibliography}{vdDM94}

\bibitem[ACZ20]{ACZBetti}
Yves Andr\'{e}, Pietro Corvaja, and Umberto Zannier.
\newblock The {B}etti map associated to a section of an abelian scheme (with an
  appendix by {Z.~Gao}).
\newblock {\em Inv. Math.}, 222:161--202, 2020.

\bibitem[BD22]{BarDill}
F.~Barroero and G.~Dill.
\newblock On the {Z}ilber--{P}ink conjecture for complex abelian varieties.
\newblock {\em Ann. Sci. \'{E}cole Norm. Sup.}, 55(1):261--282, 2022.

\bibitem[BD25]{BarDill:distinguised}
F.~Barroero and G.~Dill.
\newblock {D}istinguished categories and the {Z}ilber-{P}ink conjecture.
\newblock {\em American Journal of Mathematics}, 147(3):715--778, 2025.

\bibitem[BMZ08]{BMZUnlikely}
E.~Bombieri, D.~Masser, and U.~Zannier.
\newblock On unlikely intersections of complex varieties with tori.
\newblock {\em Acta Arith.}, 133(4):309--323, 2008.

\bibitem[BP01]{BakerPoonen:01}
M.~Baker and B.~Poonen.
\newblock Torsion packets on curves.
\newblock {\em Compositio Math.}, 127(1):109--116, 2001.

\bibitem[CMZ18]{CorvajaMasserZannier2018}
P.~Corvaja, D.~Masser, and U.~Zannier.
\newblock Torsion hypersurfaces on abelian schemes and {B}etti coordinates.
\newblock {\em Mathematische Annalen}, 371(3):1013--1045, 2018.

\bibitem[CTZ26]{CTZ:23}
P.~Corvaja, J.~Tsimerman, and U.~Zannier.
\newblock Finite {O}rbits in {S}urfaces with a {D}ouble {E}lliptic {F}ibration
  and {T}orsion {V}alues of {S}ections.
\newblock {\em Annales de l'Institut Fourier}, Online first, 2026.

\bibitem[Dav93]{DavidMinHaut}
S.~David.
\newblock Minorations de hauteurs sur les vari\'et\'es ab\'eliennes.
\newblock {\em Bulletin de la Soci\'et\'e Math\'ematique de France},
  121(4):509--544, 1993.

\bibitem[DGH21]{DGHUnifML}
V.~Dimitrov, Z.~Gao, and P.~Habegger.
\newblock Uniformity in {M}ordell--{L}ang for curves.
\newblock {\em Annals of Mathematics}, 194(1):237--298, 2021.

\bibitem[DGH22]{DGHBog}
V.~Dimitrov, Z.~Gao, and P.~Habegger.
\newblock A consequence of the relative {B}ogomolov conjecture.
\newblock {\em Journal of Number Theory (Prime), Proceedings of the First JNT
  Biennial Conference 2019}, 230:146--160, 2022.

\bibitem[DKY20]{DeMarcoKriegerYeUniManinMumford}
L.~DeMarco, H.~Krieger, and H.~Ye.
\newblock Uniform {M}anin-{M}umford for a family of genus $2$ curves.
\newblock {\em Ann. of Math.}, 191:949--1001, 2020.

\bibitem[DP02]{DPvarabII}
S.~David and P.~Philippon.
\newblock Minorations des hauteurs normalis\'ees des sous-vari\'et\'es de
  vari\'et\'es abeliennes. {II}.
\newblock {\em Comment. Math. Helv.}, 77(4):639--700, 2002.

\bibitem[ES25]{EterovicScanlon}
S.~Eterovi\'{c} and T.~Scanlon.
\newblock Likely intersections.
\newblock {\em Forum of Mathematics, Sigma}, 13(e199), 2025.

\bibitem[Fal83]{Faltings:ES}
G.~Faltings.
\newblock {E}ndlichkeitss{\"a}tze f{\"u}r abelsche {V}ariet{\"a}ten {\"u}ber
  {Z}ahlk{\"o}rpern.
\newblock {\em Invent. Math.}, 73:349--366, 1983.

\bibitem[FC90]{FaltingsChai}
G.~Faltings and C.-L. Chai.
\newblock {\em Degeneration of abelian varieties}, volume~22 of {\em Ergebnisse
  der Mathematik und ihrer Grenzgebiete (3) [Results in Mathematics and Related
  Areas (3)]}.
\newblock Springer-Verlag, Berlin, 1990.
\newblock With an appendix by David Mumford.

\bibitem[Gao20a]{GaoBettiRank}
Z.~Gao.
\newblock Generic rank of {B}etti map and unlikely intersections.
\newblock {\em Compos. Math.}, 156(12):2469--2509, 2020.

\bibitem[Gao20b]{GaoMixedAS}
Z.~Gao.
\newblock Mixed {A}x-{S}chanuel for the universal abelian varieties and some
  applications.
\newblock {\em Compos. Math.}, 156(11):2263--2297, 2020.

\bibitem[Gao21]{GaoHDR}
Z.~Gao.
\newblock Distribution of points on varieties: various aspects and
  interactions.
\newblock HDR (Habilitation \`{a} Diriger des Recherches), Sorbonne
  Universit\'{e}, 2021.

\bibitem[GH25]{GH:nondeg}
Z.~Gao and P.~Habegger.
\newblock Degeneracy {L}oci in the {U}niversal {F}amily of {A}belian
  {V}arieties.
\newblock {\em Journal of Number Theory (Prime), Proceedings of the Second JNT
  Biennial Conference 2022}, 270:96--121, 2025.

\bibitem[GN06]{GenestierNgo}
A.~Genestier and B.C. Ng\^{o}.
\newblock Lecture on {S}himura varieties.
\newblock {\em https://www.math.uchicago.edu/~ngo/Shimura.pdf}, 2006.

\bibitem[GR14a]{GR:Isogeny}
\'{E}. Gaudron and G.~R\'{e}mond.
\newblock {Polarisations et isog\'{e}nies}.
\newblock {\em Duke Math. J.}, 163(11):2057 -- 2108, 2014.

\bibitem[GR14b]{GR:periodsisgoenies}
\'{E}. Gaudron and G.~R\'{e}mond.
\newblock Th\'{e}or\`eme des p\'{e}riodes et degr\'{e}s minimaux
  d'isog\'{e}nies.
\newblock {\em Comment. Math. Helv.}, 89(2):343--403, 2014.

\bibitem[GR25]{GRGaloisBound}
E.~Gaudron and G.~R{\'e}mond.
\newblock Nombre de petits points sur une vari\'{e}t\'{e} ab\'{e}lienne.
\newblock {\em J. Inst. Math. Jussieu}, 24(3):705--761, 2025.

\bibitem[Hab13]{hab:weierstrass}
P.~Habegger.
\newblock Torsion points on elliptic curves in {W}eierstrass form.
\newblock {\em Ann. Sc. Norm. Super. Pisa Cl. Sci. (5)}, 12(3):687--715, 2013.

\bibitem[HP16]{HabeggerPilaENS}
P.~Habegger and J.~Pila.
\newblock O-minimality and certain atypical intersections.
\newblock {\em Ann. Sci. École Norm. Sup.}, 49:813--858, 2016.

\bibitem[K{\"u}h21]{KuehneUnifMM}
L.~K{\"u}hne.
\newblock Equidistribution in families of abelian varieties and uniformity.
\newblock {\em arXiv: 2101.10272}, 2021.

\bibitem[K{\"u}h23]{KuehneRBC}
L.~K{\"u}hne.
\newblock The relative {B}ogomolov conjecture for fibered products of elliptic
  curves.
\newblock {\em J.Reine Angew. Math (Crelle)}, 2023:243--270, 2023.

\bibitem[Mas84]{Masser:smallvalues}
D.~W. Masser.
\newblock Small values of the quadratic part of the {N}\'eron-{T}ate height on
  an abelian variety.
\newblock {\em Compositio Math.}, 53(2):153--170, 1984.

\bibitem[Mas89]{masser1989specializations}
D.~Masser.
\newblock Specializations of finitely generated subgroups of abelian varieties.
\newblock {\em Trans. Amer. Math. Soc.}, 311(1):413--424, 1989.

\bibitem[Mum70]{MumfordAbVar70}
D.~Mumford.
\newblock {\em Abelian Varieties}.
\newblock Oxford University Press, London, 1970.

\bibitem[MW93]{MW:abelianisog}
D.W. Masser and G.~W{\"u}stholz.
\newblock Isogeny {E}stimates for {A}belian {V}arieties, and {F}initeness
  {T}heorems.
\newblock {\em Ann. of Math. (2)}, 137(3):459--472, 1993.

\bibitem[MZ08]{MZ:torsionanomalous}
D.W. Masser and U.~Zannier.
\newblock Torsion anomalous points and families of elliptic curves.
\newblock {\em Comptes Rendus Mathematique}, 346(9):491--494, 2008.

\bibitem[MZ12]{MasserZannierTorsionPointOnSqEC}
D.~Masser and U.~Zannier.
\newblock Torsion points on families of squares of elliptic curves.
\newblock {\em Mathematische Annalen}, 352(2):453--484, 2012.

\bibitem[MZ14]{MASSER2014116}
D.~Masser and U.~Zannier.
\newblock Torsion points on families of products of elliptic curves.
\newblock {\em Advances in Mathematics}, 259:116 -- 133, 2014.

\bibitem[MZ15]{MasserZannierRelMMSimpleSur}
D.~Masser and U.~Zannier.
\newblock Torsion points on families of simple abelian surfaces and {P}ell's
  equation over polynomial rings (with an appendix by {E}. {V}. {F}lynn).
\newblock {\em Journal of the European Mathematical Society}, 17:2379--2416,
  2015.

\bibitem[MZ20]{MasserZannierRMMoverCurve}
D.~Masser and U.~Zannier.
\newblock Torsion points, {P}ell's equation, and integration in elementary
  terms.
\newblock {\em Acta Mathematica}, 225(2):227--312, 2020.

\bibitem[Paz12]{Pazuki:12}
F.~Pazuki.
\newblock Theta height and {F}altings height.
\newblock {\em Bull. Soc. Math. France}, 140(1):19--49, 2012.

\bibitem[Pin05]{Pink}
R.~Pink.
\newblock {A} {C}ommon {G}eneralization of the {C}onjectures of
  {A}ndr\'e-{O}ort, {M}anin-{M}umford, and {M}ordell-{L}ang.
\newblock {\em Preprint}, page 13pp, 2005.

\bibitem[PS13]{PeterzilStarchenkoDefinabilityTheta}
Y.~Peterzil and S.~Starchenko.
\newblock Definability of restricted theta functions and families of abelian
  varieties.
\newblock {\em Duke Math. J.}, 162(4):731--765, 2013.

\bibitem[PZ08]{PilaZannier}
J.~Pila and U.~Zannier.
\newblock Rational points in periodic analytic sets and the {M}anin-{M}umford
  conjecture.
\newblock {\em Atti Accad. Naz. Lincei Cl. Sci. Fis. Mat. Natur. Rend. Lincei
  (9) Mat. Appl.}, 19(2):149--162, 2008.

\bibitem[R{\'e}m18]{RemondGaloisBound}
G.~R{\'e}mond.
\newblock {Conjectures uniformes sur les vari\'{e}t\'{e}s ab\'{e}liennes}.
\newblock {\em The Quarterly Journal of Mathematics}, 69(2):459--486, 2018.

\bibitem[Sil92]{Silverberg:fielddef}
A.~Silverberg.
\newblock Fields of definition for homomorphisms of abelian varieties.
\newblock {\em J. Pure Appl. Algebra}, 77(3):253--262, 1992.

\bibitem[Sto17]{Stoll:simtorsion}
M.~Stoll.
\newblock Simultaneous torsion in the {L}egendre family.
\newblock {\em Exp. Math.}, 26(4):446--459, 2017.

\bibitem[Sto19]{Stoll:Uniform}
M.~Stoll.
\newblock Uniform bounds for the number of rational points on hyperelliptic
  curves of small {M}ordell-{W}eil rank.
\newblock {\em J. Eur. Math. Soc. (JEMS)}, 21:923--956, 2019.

\bibitem[vdDM94]{DriesMiller:94}
L.~van~den {D}ries and C.~Miller.
\newblock On the real exponential field with restricted analytic functions.
\newblock {\em Israel J. Math.}, 85(1-3):19--56, 1994.

\bibitem[Wil96]{Wilkie:96}
A.J. Wilkie.
\newblock Model completeness results for expansions of the ordered field of
  real numbers by restricted {P}faffian functions and the exponential function.
\newblock {\em J. Amer. Math. Soc.}, 9(4):1051--1094, 1996.

\bibitem[Yua26]{YuanArithBig}
X.~Yuan.
\newblock Arithmetic bigness and a uniform {B}ogomolov-type result.
\newblock {\em Annals of Mathematics}, 203:15--119, 2026.

\bibitem[YZ26]{YuanZhangEqui}
X.~Yuan and S.~Zhang.
\newblock {\em Adelic line bundles on quasi-projective varieties}.
\newblock Annals of Mathematical Studies. Princeton University Press, 2026.

\bibitem[Zan12]{ZannierBook}
U.~Zannier.
\newblock {\em Some problems of unlikely intersections in arithmetic and
  geometry}, volume 181 of {\em Annals of Mathematics Studies}.
\newblock Princeton University Press, Princeton, NJ, 2012.
\newblock With appendixes by David Masser.

\bibitem[Zha98]{zhang1998small}
S.~Zhang.
\newblock Small points and {A}rakelov theory.
\newblock In {\em Proceedings of the {I}nternational {C}ongress of
  {M}athematicians. {V}olume {II}}, pages 217--225, 1998.

\end{thebibliography}
